\PassOptionsToPackage{capitalize}{cleveref} 
\documentclass{article}

\usepackage{arxiv} 
\usepackage[utf8]{inputenc} 
\usepackage[T1]{fontenc}    
\usepackage{hyperref}       
\usepackage{url}            
\usepackage{booktabs}       
\usepackage{amsfonts}       
\usepackage{nicefrac}       
\usepackage{microtype}      
\usepackage{graphicx}       
\usepackage{doi}            

 \usepackage{etex}
\usepackage{lipsum}
\usepackage{amsfonts}
\usepackage{graphicx}
\usepackage{epstopdf} 
\usepackage{enumerate} 
 \usepackage{subfig,bbm,enumerate,multirow,framed,blkarray}

\usepackage{verbatim,tcolorbox}
  \usepackage{tikz}
\usetikzlibrary{shapes,arrows}
\usetikzlibrary{positioning}  
\usepackage{pstricks-add} 
 \usepackage{wrapfig}
 \usepackage{pgfplots}
  
\usepackage{color}
\usepackage{savesym}

\newcommand{\conv}{\text{conv}}

\newcommand{\xnot}{\x_0}

\newcommand{\ynot}{\mathbf{y}_0}

\newcommand{\coord}{\mathbf{e}} 
\newcommand{\E}{\mathbf{E}} 

\newcommand{\real}{\mathbb R}

\newcommand{\zz}{\mathbf{z}}
\newcommand{\w}{\mathbf{w}}

\newcommand{\uu}{\mathbf{u}}
\newcommand{\vv}{\mathbf{v}}

\newcommand{\f}{\mathbf{f}}

\newcommand{\F}{\mathbf{F}}

\newcommand{\g}{\mathbf{g}}
\newcommand{\dd}{\mathbf{d}}
\newcommand{\h}{\mathbf{h}}
\newcommand{\x}{\mathbf{x}}

\newcommand{\y}{\mathbf{y}}
\newcommand{\M}{\mathbf{M}}
\newcommand{\I}{\mathbf{I}}

\newcommand{\II}{\mathbf{I}}

\newcommand{\GG}{\mathbf{G}}
\newcommand{\A}{\mathbf{A}}

\newcommand{\m}{\mathbf{m}}
\newcommand{\N}{\mathbf{N}}

\newcommand{\rr}{\mathbf{r}}

\newcommand{\capu}{\mathbf{U}}
\newcommand{\capv}{\mathbf{V}}

\newcommand{\zero}{\mathbf{0}}

\newcommand{\Jf}{\mathbf{Jf}}
\newcommand{\J}{\mathbf{J}}
\newcommand{\JL}{\mathbf{J}^{\rm L}}

\newenvironment{bsmallmatrix}{\left[\begin{smallmatrix}}{\end{smallmatrix}\right]}

\usepackage{amsmath,amsthm,amssymb}

\usepackage{bm}
\usepackage{amsfonts}
\usepackage{enumitem}
\usepackage[mathscr]{euscript}
\usepackage{xcolor}
\usepackage{url}
\usepackage{soul}
\usepackage{subcaption}
\usepackage{graphicx}
\usepackage[skip=10pt plus 2pt minus 2pt]{parskip}

\usepackage{subfiles}
\usepackage{cite}
\usepackage{amsmath,amssymb,amsfonts}
\hypersetup{
    colorlinks=true,  
    linkcolor=black,  
    citecolor=blue,  
    urlcolor=black    
}
\usepackage{textcomp}
\usepackage{float}

\newcommand{\bgHat}{\hat{\mathbf{g}}}

\newtheoremstyle{spacedstyle}
  {12pt plus 3pt minus 3pt}
  {12pt plus 3pt minus 3pt}
  {\itshape}
  {}
  {\bfseries}
  {.}
  { }
  {}

\theoremstyle{spacedstyle}

\newtheorem{theorem}{Theorem}[section]
\newtheorem{lemma}[theorem]{Lemma}
\newtheorem{corollary}[theorem]{Corollary}
\newtheorem{proposition}[theorem]{Proposition}
\newtheorem{definition}[theorem]{Definition}
\newtheorem{remark}[theorem]{Remark}

\newtheorem{example}[theorem]{Example}

\usepackage{cleveref}
\usepackage[capitalize]{cleveref} 
\AddToHook{env/lemma/begin}{\crefalias{theorem}{lemma}}
\AddToHook{env/corollary/begin}{\crefalias{theorem}{corollary}}
\AddToHook{env/proposition/begin}{\crefalias{theorem}{proposition}}
\AddToHook{env/definition/begin}{\crefalias{theorem}{definition}}
\AddToHook{env/remark/begin}{\crefalias{theorem}{remark}}
\AddToHook{env/assumption/begin}{\crefalias{theorem}{assumption}}
\AddToHook{env/example/begin}{\crefalias{theorem}{example}}

\Crefname{theorem}{Theorem}{Theorems}
\Crefname{lemma}{Lemma}{Lemmas}
\Crefname{corollary}{Corollary}{Corollaries}
\Crefname{proposition}{Proposition}{Propositions}
\Crefname{definition}{Definition}{Definitions}
\Crefname{remark}{Remark}{Remarks}
\Crefname{assumption}{Assumption}{Assumptions}
\Crefname{example}{Example}{Examples}

\usepackage[toc,page]{appendix}

\makeatletter
\def\@maketitle{%
  \vbox to 2.5in{%
    \hsize\textwidth
    \linewidth\hsize
    \centering
    {\LARGE\bf \@title \par}
    \vskip 2em
    {\large
      \begin{tabular}[t]{c}%
        \@author
      \end{tabular}\par}
    \vfil
  }%
}
\makeatother

\begin{document}
\title{Nonsmooth High-Order Averaging Theory with Application to  Extremum Seeking Optimization and Control}

\author{Hesham Abdelfattah\thanks{Hesham Abdelfattah is with the Department of Aerospace
Engineering and Engineering Mechanics and the Department of
Mathematical Sciences, University of Cincinnati, Cincinnati, OH 45221
USA (e-mail: abdelfhm@mail.uc.edu).}
\and Sameh A. Eisa\thanks{Sameh A. Eisa is with the Department of Aerospace Engineering
and Engineering Mechanics, University of Cincinnati, OH 45221 USA
(e-mail: eisash@ucmail.uc.edu).}
\and Peter Stechlinski\thanks{Peter Stechlinski is with the Department of Mathematics
and Statistics, University of Maine, ME 04469 USA (e-mail:
peter.stechlinski@maine.edu).}}
\maketitle
\vskip -6em

\begin{abstract}
In this paper, we introduce a higher-order averaging theory and method for a wide range of nonsmooth systems that are generally characterized by the classical averaging canonical form. Utilizing tools from generalized derivatives theory, we provide a nonsmooth near-identity transformation analogous to the one in smooth averaging theory. Additionally,  we exploit sharp calculus rules from lexicographic differentiation theory to provide a closed formula for nonsmooth first-order averaging, and for the first time in the literature, nonsmooth second-order averaging. In fact, our approach recovers the smooth averaging results, without needing to check, if the system under consideration is smooth. Equipped with a nonsmooth second-order averaging theory, we generalize literature results and introduce a class of control-affine extremum seeking systems that tolerate nonsmoothness in the vector fields and/or the objective function by analyzing its stability based on a closed formula analogous to first-order Lie bracket approximations available in the smooth literature. We provide numerical simulation results involving complicated nonsmooth functions to demonstrate the effectiveness of our approach.
\end{abstract}

\section{Introduction}\label{sec:intro}
\subsection{Averaging of smooth and nonsmooth systems}\label{subsec:HigherOrderAvgIntro}
Averaging theory and methods
\cite{sanders2007averaging,maggia2020higher,pokhrel2023higher,eisa2025revisiting} 
are indispensable in studying and analyzing linear and nonlinear time-varying (especially periodic) dynamical systems. They are also crucial in advancing control theory, methods and designs for many classes of systems \cite{pokhrel2023higher,WeshengAveragingHighlyOscillatory1997,elgohary2025extremum,bullo2002averaging,ariyur2003real,abdelgalil2023multi,oliveira2022extremum}.
 In this paper, we focus on nonlinear time-varying systems that can be written in the averaging-canonical form:
\begin{align}\label{eq:average-canonical form_intro}
    \dot{\x}=\epsilon \f(\x,t,\epsilon),
\end{align} 
where $\x(t;\epsilon)\in\real^n$ is the state space vector, $t\in\real$ is time, and $0<\epsilon\ll1$ is a small parameter.
In particular, we consider \eqref{eq:average-canonical form_intro} for nonlinear time-periodic (NLTP) systems (i.e., $\f(\x,t,\epsilon)$ is $T-$periodic), which appear in many applications across the sciences and engineering fields (see \cite{maggia2020higher,pokhrel2023higher} and references therein for more details). 

In the aforementioned works 
(especially for higher-order averaging beyond first-order), it is often the case that systems in the form \eqref{eq:average-canonical form_intro} are considered with the right-hand side (RHS) function, $\f(\x,t,\epsilon)$, assumed to be smooth in $\x$ and Lebesgue integrable with respect to $t$. 
However, the smoothness assumption in $\x$ can be restrictive for many classes of systems and applications that possess nonsmooth RHS functions with respect to $\x$. Hence, multiple works have considered the challenging task of relaxing said smoothness condition.
The authors in \cite{scheinker2014non} considered a class of time-varying control-affine systems relevant to extremum seeking applications that are nondifferentiable (but only at the origin) and extended a Lie bracket averaging analysis to  provide stability results for the considered time-varying class. It is important to emphasize that the class of time-varying systems considered in \cite{scheinker2014non} can be written in the averaging canonical form \eqref{eq:average-canonical form_intro} and Lie bracket averaging is a form of higher-order averaging methods as shown in \cite{pokhrel2023higher}. The authors in \cite{iannelli2006averaging} attempted to address the effects of nonsmoothness in a class of time-varying control-affine systems via leveraging certain conditions on the dither signals (they did not address directly systems in the form \eqref{eq:average-canonical form_intro}).  
In the literature there are some relevant/notable works related to hybrid dynamical systems, where nonsmoothness with respect to $\x$ is naturally an important aspect of consideration. For instance, the authors in \cite{teel2010averaging,wang2012analysis} extended averaging methods analogous to smooth first-order averaging to analyze some classes of time-varying hybrid systems, including some stability results. Recently, the authors in \cite{abdelgalil2025lie} introduced a higher-order averaging method, extending some results from \cite{teel2010averaging} to analyze the stability of particular classes of high-frequency, periodic hybrid systems.

\subsection{Smooth and nonsmooth extremum seeking systems}\label{subsec:NonsmoothESCIntro}
Extremum seeking control (ESC) \cite{ariyur2003real,pokhrel2023higher} is a powerful model-free, real-time dynamic optimization and control method. ESC techniques aim to stabilize a dynamical system around the extremum of an
objective function, for which we have access to its measurements, but not necessarily its expression. Averaging theory and methods are instrumental in the development, analysis and design of ESC techniques -- see \cite{scheinker2024100} for a comprehensive review. In standard ESC methods (e.g.,  \cite{ariyur2003real}), first-order averaging is usually used to support the analysis/design. For control-affine ESC techniques (e.g., \cite{durr2013lie,VectorFieldGRUSHKOVSKAYA2018,labar2019newton,Suttner_ESC_Nonholonomic_2020}), Lie bracket approximations (referred to as Lie bracket averaging sometimes) are often used for stability analysis and design. In \cite{pokhrel2023higher}, it was shown that Lie bracket approximations themselves are forms of higher-order averaging, with first-order Lie bracket approximations \cite{durr2013lie,VectorFieldGRUSHKOVSKAYA2018} being equivalent to second-order averaging and second-order Lie bracket approximation \cite{labar2019newton} being equivalent to third-order averaging. Additionally, a fourth-order averaging formula (third-order Lie bracket approximation) was provided in \cite{pokhrel2023higher} with a procedure to generate any higher-order averaging term for a general class of control-affine systems (including ESCs). It is worth highlighting that using higher-order Lie bracket approximations motivated recent results of ESC that guarantee exponential convergence even when the objective function behaves locally as a higher-order polynomial function \cite{grushkovskaya2025extremum}. 

In the aforementioned works on ESC, it is usually the case that smoothness of the RHS function with respect to $\x$ is assumed (sometimes relaxed to a $C^2$ condition) for the objective function and/or the system's dynamics to enable ESC analysis via averaging methods. In model-free methods, such as ESC techniques, it cannot be overstated to assert that any relaxation of smoothness conditions of the RHS function with respect to $\x$ is of significant importance to enable wider applications of said model-free methods. Hence, some works have utilized nonsmooth system theory and tools to relax the smoothness assumptions typically assumed in ESC literature. For control-affine ESC systems, the authors in \cite{scheinker2014non} studied a non-$C^2$ controller that minimizes the value of the squared Euclidean norm (resulting in a closed-loop system that is smooth everywhere except at the origin). The authors in \cite{suttner2023nonsmooth} aimed at expanding the premise of \cite{scheinker2014non} and utilized a stochastic version of the Lie bracket approximation approach in \cite{durr2013lie} with locally Lipschitz continuous objective functions. For a class of  ESC systems, the authors in \cite{poveda2021nonsmooth} studied ESCs  involving discontinuous vector fields by using averaging and singular perturbation tools for nonsmooth and set-valued  systems. The authors in \cite{williams2024semiglobal} utilized nonsmooth analysis tools for stability guarantees for safe extremum seeking (which achieves the minimization of an unknown objective function subject to some unknown, yet measured, constraints). In addition, recently the authors in \cite{rodrigues2025event} introduced event-triggered ESC for which they utilized Lyapunov and averaging theories for discontinuous systems for analysis.

\subsection{Motivation and contributions}\label{subsec:Motivation_Contribution}
\textbf{Motivation 1. A lack of nonsmooth averaging theory for general dynamic systems}: Higher-order averaging analysis for general dynamic systems in the canonical form \eqref{eq:average-canonical form_intro} requires first- and higher-order differentiation of the RHS function with respect to $\x$. 
This is not only important for the theory to hold, but also for the computation of higher-order averaging explicit formulas that can be used by the user in analysis/applications, similar to the smooth case \cite{sanders2007averaging,maggia2020higher,pokhrel2023higher}. As discussed in section 1.1, some developments have been made in the literature to relax smoothness conditions of the RHS function with respect to $\x$. However, the available literature of nonsmooth averaging, to the best of our knowledge, is either related to first-order averaging or dealing with a restricted class of systems (e.g., affine), not a general dynamic system in the  canonical form \eqref{eq:average-canonical form_intro}. 

\textbf{Motivation 2. Limitations in the generalized derivatives methods used in the averaging literature}: The literature on nonsmooth averaging relies on generalized derivatives theory to compensate for the lack of smoothness during analysis. 
However, the kinds of generalized derivative methods used in the nonsmooth averaging literature 
are    
computationally challenging to implement due to, for example, the lack of sharp calculus rules (e.g., sum and chain rules) as illustrated in \cite{ref:OMS_generalizedderivatives}, which curtailed the development of easy-to-compute nonsmooth higher-order averaging explicit formulas analogous to the smooth computations.

\textbf{Motivation 3. A lack of general nonsmooth control-affine extremum seeking}:  
As aspired by \cite{suttner2023nonsmooth}, it will be important to
generalize control-affine ESC techniques to   control laws and objective functions that  can be locally Lipschitz continuous (rather than $C^2$ everywhere except the origin).

\textbf{Contributions}: In this work, we aim to overcome the above-mentioned limitations by bringing, for the first time in the literature,  lexicographic differentiation theory and tools \cite{nesterov2005lexicographic,khan2015vector} to averaging theory and methods. Lexicographic differentiation theory and tools are applicable to a broad range of nonsmooth functions referred to as L-smooth (more details are given in \cref{sec:Prem}), and are computationally relevant, enabling for example, sharp calculus rules and algorithmic computations, as can be seen in their recent success addressing properties such as identifiability and observability of nonsmooth systems \cite{stechlinski2025identifiability}. Equipped with lexicographic differentiation theory and tools, we introduce a novel nonsmooth higher-order averaging theory for general nonsmooth dynamic systems in the canonical form \eqref{eq:average-canonical form_intro} by addressing the nonsmoothness in the system directly (instead of using smoothing approximations, for example). Additionally, we use said novel theory to analyze nonsmooth control-affine ESC systems.

In this paper, we achieve the following results:

\begin{itemize}
    \item In \cref{sec:Prem}, we establish a new nonsmooth concept of Lie bracket, called the lexicographic Lie bracket (L-Lie bracket), and prove some of its properties. 
    \item In \cref{sec:HigherOrderAvg}, we introduce a nonsmooth near-identity (coordinate) transformation that admits a class of nonsmooth functions to enable nonsmooth averaging theory. Then, we use lexicographic differentiation tools to provide closed formulas for first- and second-order averaging terms of nonsmooth systems. We prove that trajectories generated by first- and second-order averaged systems maintain bounds, approximation errors, and stability results analogous to smooth (classical) averaging.
    \item 
    In \cref{sec:NonsmoothESC}, we use our nonsmooth second-order averaging results from \cref{sec:HigherOrderAvg} to provide an approximation system for design and stability analysis of a class of nonsmooth ESC systems (analogous to the smooth second-order averaging, or equivalently first-order Lie bracket approximation in  \cite{pokhrel2023higher}). We provide multiple numerical simulations involving complicated nonsmooth functions to demonstrate the effectiveness of our results.  
\end{itemize}
\section{Preliminaries}\label{sec:Prem} 

\subsection{Clarke's generalized derivative}\label{subsec:genderivs}
We begin by introducing Clarke's generalized derivative  \cite{clarke1990optimization}. Consider an open set $X\subseteq \real^n$ and a function $\f:X\to\real^m$ that is locally Lipschitz on $X$. First we define the Bouligand subdifferential (B-subdifferential) of $\f$ as 
\begin{equation*}  
\partial_{\text{B}} \f(\x)
:=\left\{\F: \exists \x_i \to \x \text{ s.t. } \x_i \in X \setminus Z_{\f}, \J \f(\x_{i}) \to \F \right\},
\end{equation*}  
where  $\J\f(\x)$ is the Jacobian of $\f$ at $\x$ and $Z_{\f}$ is the zero measure subset of points where $\f$ is not differentiable. Clarke's generalized derivative of $\f$ at $\x \in X$ is given by:
\begin{equation}\label{eq:Clarke}
\partial_{\rm C}\f(\x):=\conv \; \partial_{\rm B}\f(\x),
\end{equation}
i.e., the convex hull of the B-subdifferential of $\f$. We note that the set $\partial_{\rm C}\f(\x)$ is nonempty, compact and convex \cite{clarke1990optimization}.
For example, if $f(x)=|x|$, we get: 
\begin{equation*}
\partial_{\rm C} f(x)=
  \begin{cases}
  \{-1\},& x < 0, \\
  [-1,1],& x = 0, \\
  \{1\},& x > 0. 
  \end{cases}
\end{equation*}
If $\f$ is $C^1$ at $\x$, we recover the classical derivative, i.e.,  $\partial_{\rm B} \f(\x) = \partial_{\rm C} \f(\x)=\{\Jf(\x)\}$. If we consider an open set $W\subseteq \real^n\times\real^q$ and a 
 function $\g:W\to\real^m$ that is locally Lipschitz, the  Clarke  generalized derivative  projection of $\g$ with  respect  to $\y$ at  $(\x,\y)$  is defined as \cite{clarke1990optimization}:
\begin{align}\label{eq:Clarke_Projection}
    \pi_{\y}\partial_{\rm C} \g(\x,\y) &= \{\N \in \real^{m\times q}:\exists[\M \;\;\; \N] \in\partial_{\rm C} \g(\x,\y)\}.
\end{align}

\subsection{Overview of lexicographic differentiation}\label{subsec:LD}
Nonsmooth numerical methods for 
optimization 
\cite{Stechlinski_Khan_Barton_NLP_LDDs} 
can be supplied with elements from Clarke's generalized derivative  {(see \cite{ref:OMS_generalizedderivatives} for more details)}, however, there is a computational challenge associated with  Clarke's generalized derivative calculations since it does not satisfy sharp calculus rules. In order to overcome this limitation, a framework based on lexicographic differentiation  \cite{nesterov2005lexicographic} and lexicographic directional differentiation \cite{khan2015vector} was introduced. This framework is applicable  to the class of L-smooth functions that includes all $C^1$, convex (such as any p-norm function), and $PC^1$ functions (such as the absolute value function, min and max functions) \cite{Scholtes}, and compositions thereof. More specifically, we define an L-smooth function as follows \cite{nesterov2005lexicographic}:
\begin{definition}\label{def:L-smoothFunction} 
     Given some open set $X \subseteq \mathbb{R}^n$, a function $\f:X\rightarrow \mathbb{R}^m$ is said to be L-smooth at $\x\in X$ if $\f$ is Lipschitz continuous on a neighborhood of $\x$ and directionally differentiable to arbitrary order. That is, for any $k\in \mathbb{N}$ and directions matrix $\M=[
      \m_{1} \quad \dots \quad \m_{k} ]\in \mathbb{R}^{n\times k}$, the following homogenization sequence exists: 
  \begingroup
  \allowdisplaybreaks
  \begin{align*}
\f^{(0)}&:\mathbb{R}^n\rightarrow\mathbb{R}^m:\uu\rightarrow \f'(\x;\uu),\\
\f^{(1)}&:\mathbb{R}^n\rightarrow\mathbb{R}^m:\uu\rightarrow [\f^{(0)}]'(\m_{1};\uu),\\ \f^{(2)}&:\mathbb{R}^n\rightarrow\mathbb{R}^m:\uu\rightarrow [\f^{(1)}]'(\m_{2};\uu),\\
      &\vdots\\
\f^{(k)}&:\mathbb{R}^n\rightarrow\mathbb{R}^m:\uu\rightarrow [\f^{(k-1)}]'(\m_{k};\uu),
  \end{align*}
\endgroup
where  $\f'(\x;\uu)=\lim_{\alpha\downarrow 0}\tfrac{\f(\x+\alpha\uu)-\f(\x)}{\alpha}$.
\end{definition}

Intuitively, the homogenization process aims to probe along the directions in the columns of the directions matrix $\M$ (taken from left to right), which be thought of as if we are ``zooming in'' and ``flattening'' the function at each stage. If the columns of $\M$ span the domain space $\real^n$, it is guaranteed that the final mapping $\f^{(k)}$ is linear and we are able to calculate the so-called lexicographic derivative (L-derivative) \cite{nesterov2005lexicographic}, which is the derivative of this final mapping:
\begin{equation}\label{eq:Lderivative}
\J^{\rm L}\f(\x;\M):=\J \f^{(k)}(\zero_n)\in \real^{m \times n}.
\end{equation}
The L-derivative can be considered as a Jacobian-like object for L-smooth functions 
{as it} provides derivative information that can be used for nonsmooth numerical methods {(see \cite{ref:OMS_generalizedderivatives} for details)}. It is important to note that if $f$ is a scalar-valued function, then $\nabla_{\rm L} f (\x;\M)=(\J^{\rm L}f(\x;\M))^{\rm T} \in \partial_{\rm C}f(\x)$, i.e., the gradient-like object, referred to as an ``L-gradient", is a Clarke subgradient. 

We also introduce the following notation: in the case that the directions matrix is taken to be the identity matrix, i.e., $\M=\I$, then we drop the directions matrix argument from the L-gradient and L-derivative, i.e., 
$$\nabla_{\rm L} f (\x):=\nabla_{\rm L} f (\x;{\I}), \quad \J^{\rm L}\f(\x):=\J^{\rm L}\f(\x;\I).$$ 
For a function $\f:X \times Y\subseteq \real^n \times \real^m \to \real^q$, we denote 
\begin{equation}\label{eq.Lnotation}
\begin{aligned}
\J^{\rm L}_{\x}(\x,\y)
:=\J^{\rm L}\f(\x,\y;(\I,\zero))=\J^{\rm L}[\f(\cdot,\y)](\x;\I), \qquad
\J^{\rm L}_{\y}(\x,\y)
:=\J^{\rm L}\f(\x,\y;(\zero,\I))=\J^{\rm L}[\f(\x,\cdot)](\y;\I)
\end{aligned}
\end{equation}
i.e., the L-derivative of the vector-valued function $\f$ with respect to $\x$ (with $\y$ fixed) and $\y$ (with $\x$ fixed) in the directions $\I_{n}$, respectively. The set of L-derivatives at a point is called the lexicographic subdifferential (L-subdifferential) and defined as \cite{nesterov2005lexicographic}:
\begin{equation*}
\partial_{\rm L}\f(\x):=\{\J^{\rm L}\f(\x;\M): \M\in\real^{n \times k} \text{ has full row rank}\}.
\end{equation*}
The lexicographic directional derivative (LD-derivative) \cite{khan2015vector} of an L-smooth function, which is a computationally-relevant tool that  can be used to calculate L-derivatives, is defined as \cite{khan2015vector}:
\begin{equation}\label{eq:LDderivative}
\f'(\x;\M):=[\f^{(0)}(\m_1)
\;\;
\f^{(1)}(\m_2)
\;
\cdots
\;
\f^{(k-1)}(\m_k)].
\end{equation}  
In case $\M$ is full row rank (and is hence right-invertible), the L-derivative $\J^{\rm L}\f(\x;\M)$ of $\f$ at $\x$ in the directions $\M$ can be furnished using $\f'(\x;\M)$ by the following relation: 
\begin{equation}\label{eq.LD.L.derivative}
\underbrace{\f'(\x;\M)}_{m \times k}=\underbrace{\J^{\rm L}\f(\x;\M)}_{m \times n} \underbrace{\M}_{n \times k}.
\end{equation} 
One of the main reasons for using the LD-derivative  to calculate the L-derivative is that the LD-derivative satisfies sharp calculus rules; given a  directions matrix $\M$ and L-smooth functions $\f$ and $\g$, the component-wise LD-derivative rule, the sum rule, the product rule, and the chain rule, are presented below, respectively:
\begin{equation}\label{eq:sharp_rules}
      \begin{split}
           \f '(\x;\M)&=(f_1'(\x;\M),f_2'(\x;\M),\ldots,f_m'(\x;\M)),\\
           [\f+\g]'(\x;\M)&=\f'(\x;\M)+\g'(\x;\M),\\
            [fg]'(\x;\M)&=g(\x)f'(\x;\M)+f(\x)g'(\x;\M),\\
           [\f \circ \g]'(\x;\M)&=\f'(\g(\x);\g'(\x;\M)).
      \end{split}
  \end{equation}
In addition to the sharp calculus rules mentioned above, another advantage of LD-derivatives is that they provide closed-form expressions for commonly used nonsmooth functions. For example, given a directions matrix $\M=[m_1 \quad m_2 \quad \cdots \quad m_k] \in \mathbb{R}^{1\times k}$, we have 
\begin{equation}\label{eq.LDderivative.abs}
\mathrm{abs}'(x;\M)=\mathrm{fsign}(x,m_1,m_2,\dots,m_k)\M,
\end{equation}
where first-sign function, $\text{fsign}(\cdot)$, returns the sign of the first nonzero element, or zero if its input is zero. Also, given a directions matrix $\M=\begin{bsmallmatrix} \M_1 \\ \M_2 \end{bsmallmatrix}\in \mathbb{R}^{2\times k}$, we have
\begin{align}\label{eq:min_function_LD}
    &\min{'}(u_0,v_0;\bm{\mathrm{M}}) 
 = {\rm \bf slmin}([u_0 \quad \M_1],[v_0 \quad \M_2])
                :=  \begin{cases}
                          \M_1,& \text{if } \mathrm{fsign}\left( \left( u_0 , \M_1\right) - \left(
                          v_0, \M_2 \right)\right)\le 0
                          ,\\
                          \M_2,& \text{otherwise},
                      \end{cases}
\end{align}
i.e., the function ${\bf slmin}$, called the shifted-lexicographic-minimum function, returns the lexicographic minimum of the two vector arguments, left-shifted by one element. For $C^1$ functions, the LD-derivative recovers the classical directional derivative. For example, we have $\sin'(x;\M)=\cos(x)\M$. The chain rule in \eqref{eq:sharp_rules} can  be used to provide the LD-derivatives of compositions of L-smooth functions.
For instance, if we consider the composition ${h(x)=[f \circ g](x)}=\sin{(\mathrm{abs}(x))}$, then the LD-derivative of ${h}$ at $x$ along the directions matrix $\M\in\real^{1\times k}$ is:
\begin{align*}
    {h}'(x;\M) 
    {= f'(g(x);g'(x;\M))}
    &= [\sin]'(\mathrm{abs}(x);\mathrm{abs'}(x;\M)) \\
    &=\cos(\mathrm{abs}(x))\mathrm{fsign}(x,m_1,m_2,\dots,m_k)\M.
\end{align*}
Similarly, for ${h(x)}=\cos(\max(x,0))$ along  $\M \in \mathbb{R}^{1\times k}$,
\begin{align*}
   {h}'(x;\M)
    &= [\cos]'(\max(x,0);{\max{'}(x,0;(\M,\zero_{1 \times k}))}\\ 
    &=-\sin(\max(x,0)){\rm \bf slmax}([x \quad \M],[0 \quad \zero_{1\times k}]).
\end{align*}
Also, we see that for ${h(x)}=\mathrm{exp}(\mathrm{abs}(x))$ with $\M\in\real^{1\times k}$ we have
\begin{align*}
    {h}'(x;\M) &= [\mathrm{exp}]'(\mathrm{abs}(x);\mathrm{abs'}(x;\M)) 
    =\mathrm{exp}(\mathrm{abs}(x))\mathrm{fsign}(x,m_1,m_2,\dots,m_k)\M.
\end{align*}
More details can be found in \cite{ref:OMS_generalizedderivatives,khan2015vector}. 

Next, we provide a nonsmooth Taylor-like approximation (expansion) of a function with respect to a subset of its variables, based on the theory in \cite{stechlinski2025identifiability}, for use later in the averaging analysis.

\begin{proposition}\label{prop:partial_L_Taylor}
Let $\f:X_1\times X_2\to\real^p$ be L-smooth, with $X_1\subseteq \real^n$, $X_2\subseteq \real^m$ open sets. Then, given any $\xnot\in X_1$, $\ynot\in X_2$,
\begin{equation}\label{eq:Partial_L_Taylor_LD}
        \lim_{\x\to\xnot} \frac{\|\f(\x,\ynot)-\f(\xnot,\ynot)-\f'\left(\xnot,\ynot;\begin{bsmallmatrix}
            \Delta \x & \I_n & \zero_{n\times m}\\ \zero_{m}&\zero_{m\times n}&\I_{m}
        \end{bsmallmatrix}\right)\begin{bsmallmatrix}
            0\\\Delta \x\\\zero_{m}
        \end{bsmallmatrix}\|}{\|\Delta \x\|}=0,
    \end{equation}   
where $\Delta \x = \x-\xnot$, or, equivalently,  
\begin{equation}\label{eq:Partial_L_Taylor_L}
        \lim_{\x\to\xnot} \frac{\|\f(\x,\ynot)-\f(\xnot,\ynot)-\JL\f\left(\xnot,\ynot;\begin{bsmallmatrix}
            \Delta \x & \I_n & \zero_{n\times m}\\ \zero_{m}&\zero_{m\times n}&\I_{m}
        \end{bsmallmatrix}\right)\begin{bsmallmatrix}
           \Delta \x\\\zero_{m}
        \end{bsmallmatrix}\|}{\|\Delta \x\|}=0.
    \end{equation}   

\end{proposition}

\begin{proof}
    Consider the L-smooth function $\pmb{\varphi}(\x):=\f(\x,\ynot)$.
    Let $\M=[\Delta \x \quad \I_n]$.
    We aim to show, using induction, that 
    \begin{equation}\label{eq.inductiongoal}
        \pmb{\varphi}'(\x_0;\M)\begin{bsmallmatrix}
            0\\\Delta \x
\end{bsmallmatrix}=\f'\left(\xnot,\ynot;\begin{bsmallmatrix}
        \Delta \x & \I_n & \zero_{n\times m}\\ \zero_{m}&\zero_{m\times n}&\I_{m}
    \end{bsmallmatrix}\right)\begin{bsmallmatrix}
            0\\\Delta \x\\\zero_{m}
        \end{bsmallmatrix}.
    \end{equation}
    We note that, by definition,
    \begin{align*}
        \pmb{\varphi}'(\x_0;\M)=\begin{bmatrix}
            \pmb{\varphi}^{(0)}(\Delta \x)&\pmb{\varphi}^{(1)}(\coord_1)&\cdots&\pmb{\varphi}^{(n)}(\coord_n)
        \end{bmatrix},
    \end{align*}
    and
       \begin{align*}
        &\f'\left(\xnot,\ynot;\begin{bsmallmatrix}
        \Delta \x & \I_n & \zero_{n\times m}\\ \zero_{m}&\zero_{m\times n}&\I_{m}
    \end{bsmallmatrix}\right) \\
       &=[
            \f^{(0)}(\Delta \x,\zero_{m})
            \quad
            \f^{(1)}(\coord_1,\zero_{m})
            \quad
            \cdots
            \quad
            \f^{(n)}(\coord_n,\zero_{m})
            \quad
            \f^{(n+1)}(\zero_{n},\tilde{\coord}_{1})
           \quad
            \cdots
            \quad
            \f^{(n+m)}(\zero_{n},\tilde{\coord}_{m})
            ],
    \end{align*}
    where $\coord_i$, $i=1,\dots,n$, and $\tilde{\coord}_j$, $j=1,\dots,m$  are the standard basis vectors in $\real^n$ and $\real^m$, respectively. Hence, to show \eqref{eq.inductiongoal}, we only need show that
    \begin{align*}
     \pmb{\varphi}^{(0)}(\Delta \x)&=\f^{(0)}(\Delta \x,\zero_{m}),\quad
          \pmb{\varphi}^{(i)}(\coord_i)=\f^{(i)}(\coord_i,\zero_{m}), \quad i=1,\ldots,n.
    \end{align*}
    Let $\dd\in\real^n$. First, note that we have 
    \begin{align*}
        \pmb{\varphi}^{(0)}(\dd)&=\pmb{\varphi}'(\xnot;\dd)=\lim_{\alpha\downarrow0}\frac{\pmb{\varphi}(\xnot+\alpha\dd)-\pmb{\varphi}(\xnot)}{\alpha}=\lim_{\alpha\downarrow0}\frac{\f(\xnot+\alpha\dd,\ynot)-\f(\xnot,\ynot)}{\alpha}\\
        &=\lim_{\alpha\downarrow0}\frac{\f(\xnot+\alpha\dd,\ynot+\alpha\zero_{m})-\f(\xnot,\ynot)}{\alpha}
        =\f'(\xnot,\ynot;(\dd,\zero_{m}))
        =\f^{(0)}(\dd,\zero_{m}).
    \end{align*}
    Now we proceed by induction to prove $\pmb{\varphi}^{(i)}(\dd)=\f^{(i)}(\dd,\zero_{m})$, for $i=1,\dots,n$. For the base case, we have
        \begin{align*}
        \pmb{\varphi}^{(1)}(\dd)&=[\pmb{\varphi}^{(0)}]'(\m_1;\dd)=\lim_{\alpha\downarrow0}\frac{\pmb{\varphi}^{(0)}(\m_1+\alpha\dd)-\pmb{\varphi}^{(0)}(\m_1)}{\alpha}
        =\lim_{\alpha\downarrow0}\frac{\f^{(0)}(\coord_1+\alpha\dd,\zero_{m})-\f^{(0)}(\coord_1,\zero_{m})}{\alpha}\\
        &=\lim_{\alpha\downarrow0}\frac{\f^{(0)}(\coord_1+\alpha\dd,\zero_{m}+\alpha\zero_{m})-\f^{(0)}(\coord_1,\zero_{m})}{\alpha}
        =[\f^{(0)}]'(\coord_1,\zero_{m};(\dd,\zero_{m}))
        =\f^{(1)}(\dd,\zero_{m}).
    \end{align*}
    For the inductive step, suppose $\pmb{\varphi}^{(k)}(\dd)=\f^{(k)}(\dd,\zero_{m})$, then
    \begin{align*}
        \pmb{\varphi}^{(k+1)}(\dd)&=[\pmb{\varphi}^{(k)}]'(\m_k;\dd)=\lim_{\alpha\downarrow0}\frac{\pmb{\varphi}^{(k)}(\m_k+\alpha\dd)-\pmb{\varphi}^{(k)}(\m_k)}{\alpha}
        =\lim_{\alpha\downarrow0}\frac{\f^{(k)}(\coord_k+\alpha\dd,\zero_{m})-\f^{(k)}(\coord_k,\zero_{m})}{\alpha}\\
        &=\lim_{\alpha\downarrow0}\frac{\f^{(k)}(\coord_k+\alpha\dd,\zero_{m}+\alpha\zero_{m})-\f^{(k)}(\coord_k,\zero_{m})}{\alpha}
        =[\f^{(k)}]'(\coord_k,\zero_{m};(\dd,\zero_{m}))
        =\f^{(k+1)}(\dd,\zero_{m}).
    \end{align*}
Hence, we see that for any fixed $\dd\in \real^n$, $\pmb{\varphi}^{(i)}(\dd)=\f^{(i)}(\dd,\zero_{m})$ for $i=0,\dots,n$. Then, we get
\begin{align*}
    \pmb{\varphi}'(\x_0;\M)\begin{bsmallmatrix}
            0\\\Delta \x
\end{bsmallmatrix}
&= [
            \pmb{\varphi}^{(0)}(\Delta \x)\quad\pmb{\varphi}^{(1)}(\coord_1)\quad\dots\quad\pmb{\varphi}^{(n)}(\coord_n)
        ]\begin{bsmallmatrix}
            0\\\Delta \x
\end{bsmallmatrix}
        =[\pmb{\varphi}^{(1)}(\coord_1 )\quad\pmb{\varphi}^{(2)}(\coord_2)\quad\dots\quad\pmb{\varphi}^{(n)}(\coord_n)
        ]\Delta \x \\
        &=[\f^{(1)}(\coord_1 ,\zero_{m})\quad\f^{(2)}(\coord_2,\zero_{m})\quad\dots\quad\f^{(n)}(\coord_n,\zero_{m})
        ]\Delta \x \\
        &=[\f^{(0)}(\Delta \x ,\zero_{m}),\f^{(1)}(\coord_1 ,\zero_{m})\quad\dots\quad\f^{(n)}(\coord_n,\zero_{m})\quad\f^{(n+1)}(\zero_{n},\tilde{\coord}_1)\quad\dots\quad\f^{(n+m)}(\zero_{n},\tilde{\coord}_m)]\begin{bsmallmatrix}
            0\\\Delta \x\\\zero_{m}
        \end{bsmallmatrix}\\
        &=\f'\left(\xnot,\ynot;\begin{bsmallmatrix}
        \Delta \x & \I_n & \zero_{n\times m}\\ \zero_{m}&\zero_{m\times n}&\I_{m}
    \end{bsmallmatrix}\right)\begin{bsmallmatrix}
            0\\\Delta \x\\\zero_{m}
        \end{bsmallmatrix}.
\end{align*}
From \cite[Theorem 3.1]{stechlinski2025identifiability}, we have that
\begin{align*}
    \lim_{\x\to\x_0} \frac{\|\pmb{\varphi}(\x)-\pmb{\varphi}(\x_0)-\pmb{\varphi}'(\x_0;\M)\begin{bsmallmatrix}
            0 \\\Delta \x 
        \end{bsmallmatrix}\|}{\|\Delta \x\|}=0,
\end{align*}
and \eqref{eq:Partial_L_Taylor_LD} follows. For the L-derivative form, we observe
\begin{align*}
    \begin{bmatrix}
            \Delta \x & \I_n & \zero_{n\times m}\\ \zero_{m}&\zero_{m\times n}&\I_{m}
        \end{bmatrix}\begin{bmatrix}
            \zero_{1\times n} & \zero_{1\times m} \\ \I_{n} & \zero_{n\times m}\\ \zero_{m\times n}&\I_{m}
        \end{bmatrix} =   \begin{bmatrix}
            \I_n & \zero_{n\times m}\\ \zero_{m\times n}&\I_{m}
        \end{bmatrix},
\end{align*}
i.e., $\begin{bsmallmatrix}
            \zero_{1\times n} & \zero_{1\times m} \\ \I_{n} & \zero_{n\times m}\\ \zero_{m\times n}&\I_{m}
        \end{bsmallmatrix}$ is the right inverse of $\begin{bsmallmatrix}
            \Delta \x & \I_n & \zero_{n\times m}\\ \zero_{m}&\zero_{m\times n}&\I_{m}
        \end{bsmallmatrix}$. Hence, using \eqref{eq.LD.L.derivative}, we get
\begin{align*}
    &\f'(\xnot,\ynot;\begin{bsmallmatrix}
            \Delta \x & \I_n & \zero_{n\times m}\\ \zero_{m}&\zero_{m\times n}&\I_{m}
        \end{bsmallmatrix})\begin{bsmallmatrix}
            0\\\Delta \x\\\zero_{m}
        \end{bsmallmatrix}= \f'(\xnot,\ynot;\begin{bsmallmatrix}
            \Delta \x & \I_n & \zero_{n\times m}\\ \zero_{m}&\zero_{m\times n}&\I_{m}
        \end{bsmallmatrix})\begin{bsmallmatrix}
            \zero_{1\times n} & \zero_{1\times m} \\ \I_{n} & \zero_{n\times m}\\ \zero_{m\times n}&\I_{m}
        \end{bsmallmatrix}\begin{bsmallmatrix}
            \Delta \x\\\zero_{m}
        \end{bsmallmatrix}\\
        &= \f'(\xnot,\ynot;\begin{bsmallmatrix}
            \Delta \x & \I_n & \zero_{n\times m}\\ \zero_{m}&\zero_{m\times n}&\I_{m}
        \end{bsmallmatrix})\begin{bsmallmatrix}
            \Delta \x & \I_n & \zero_{n\times m}\\ \zero_{m}&\zero_{m\times n}&\I_{m}
        \end{bsmallmatrix}^{-1}\begin{bsmallmatrix}
            \Delta \x\\\zero_{m}
        \end{bsmallmatrix}
        =\JL\f\left(\xnot,\ynot;\begin{bsmallmatrix}
            \Delta \x & \I_n & \zero_{n\times m}\\ \zero_{m}&\zero_{m\times n}&\I_{m}
        \end{bsmallmatrix}\right) \begin{bsmallmatrix}
            \Delta \x\\\zero_{m}
        \end{bsmallmatrix}.
\end{align*}
Finally, we see from \eqref{eq:Partial_L_Taylor_LD} that 
\begin{align*}
    \f(\x,\ynot) = \f(\xnot,\ynot)+\f'\left(\xnot,\ynot;\begin{bsmallmatrix}
            \Delta \x & \I_n & \zero_{n\times m}\\ \zero_{m}&\zero_{m\times n}&\I_{m}
        \end{bsmallmatrix}\right)\begin{bsmallmatrix}
            0\\\Delta \x\\\zero_{m}
        \end{bsmallmatrix}+\mathcal{O}(\|\begin{bsmallmatrix}
            0\\\Delta \x\\\zero_{m}
        \end{bsmallmatrix}\|^2),
\end{align*}
and \eqref{eq:Partial_L_Taylor_L} follows.
\end{proof}

\begin{remark}
From \eqref{eq:Partial_L_Taylor_LD} and \eqref{eq:Partial_L_Taylor_L}, we have that  
    \begin{align}
        \f(\x,\ynot)
        &=\f(\xnot,\ynot)+\f'\left(\xnot,\ynot;\begin{bsmallmatrix}
            \Delta \x & \I_n & \zero_{n\times m}\\ \zero_{m}&\zero_{m\times n}&\I_{m}
        \end{bsmallmatrix}\right)\begin{bsmallmatrix}
            0 \\ \Delta \x\\\zero_{m}
        \end{bsmallmatrix} + \mathcal{O}(\|
            \Delta \x\|^2)\label{eq:Partial_LD_Taylor_Approx}\\
     &=\f(\xnot,\ynot)+\JL\f\left(\xnot,\ynot;\begin{bsmallmatrix}
            \Delta \x & \I_n & \zero_{n\times m}\\ \zero_{m}&\zero_{m\times n}&\I_{m}
        \end{bsmallmatrix}\right)\begin{bsmallmatrix}
            \Delta \x\\\zero_{m}
        \end{bsmallmatrix} + \mathcal{O}(\|
            \Delta \x\|^2)\label{eq:Partial_L_Taylor_Approx}
    \end{align}
Moreover, if $\f$ is $C^1$ at $(\xnot,\ynot)$, then $\JL\f=\J\f$ and this simplifies as
 \begin{align*}
        \f(\x,\ynot)
        &=\f(\xnot,\ynot)+\J\f(\xnot,\ynot)\begin{bsmallmatrix}
            \Delta \x\\\zero_{m}
        \end{bsmallmatrix} + \mathcal{O}(\|
            \Delta \x\|^2)
     =\f(\xnot,\ynot)+\frac{\partial \f}{\partial \x}(\xnot,\ynot) \Delta \x+ \mathcal{O}(\|
            \Delta \x\|^2),
                \end{align*}
                as expected.
\end{remark}

    
\begin{example}\label{ex:firstorder_TAC}
This example is adapted from \cite[Example 3.3]{stechlinski2025identifiability}. 
Consider the function $f(x,y)=|x^2-y^2|$ and let $(x_0,y_0)=(1,1)$. Consider $g(x,y)=x^2-y^2$, and take 
 $\M=\begin{bsmallmatrix}
            \Delta \x & \I_n & \zero_{n\times m}\\ \zero_{m}&\zero_{m\times n}&\I_{m}
        \end{bsmallmatrix}=\begin{bsmallmatrix} x-1 & 1 & 0 \\ 0 & 0 & 1 \end{bsmallmatrix}$, i.e., $n=m=1$.

Since $g$ is a $C^1$ function, then we have, 
\begin{align*}g'(x_0,y_0;\M)
=\J g(x_0,y_0) \M
&=\begin{bmatrix}
     2 & -2
\end{bmatrix}\begin{bmatrix}
   x-1 & 1 & 0 \\ 0 & 0 & 1
\end{bmatrix}=\begin{bmatrix}
    2(x-1) & 2 & -2
\end{bmatrix}.
\end{align*}
Then, from the chain rule in \eqref{eq:sharp_rules}, and the formula in \eqref{eq.LDderivative.abs}, we get
\begin{align*}
f'(x_0,y_0;\M)&=
{\rm abs}^{'}(g(x_0,y_0);g'(x_0,y_0;\M))=\mathrm{fsign}(0,2(x-1),2,-2) \times [2(x-1) \quad 2 \quad -2].
\end{align*}
Hence, we get the following from \eqref{eq:Partial_LD_Taylor_Approx}
\begin{align*}
f(x,y_0) 
\approx f(x_0,y_0)+f'(x_0,y_0;\begin{bsmallmatrix} x-1 & 1 & 0 \\ 0 & 0 & 1 \end{bsmallmatrix}) \begin{bsmallmatrix} 0  \\ x-1 \\ 0 \end{bsmallmatrix}
&=\mathrm{fsign}(0,2(x-1),2,-2)  \times (2(x-1))\\
&=\begin{cases} 
2(x-1) & \text{if } x \geq 1, \\ 
-2(x-1)& \text{if } x<1. \end{cases}
\end{align*}
 \end{example}

\subsection{Nonsmooth and Discontinuous ODE Systems}\label{subsec:filippov}
Consider the non-autonomous ODE system
\begin{equation}\label{eq.ODEsystem}
\dot{\x} = \f(\x,t), \quad \x(t_0) = \x_0,
\end{equation}
where $\f: X \times \real \to \real^n$, with $X \subseteq \real^n$ open and connected. In the case that $\f$ is continuous, it follows from classical theory that \eqref{eq.ODEsystem} admits a classical  solution (i.e., a $C^1$ solution) local to $t_0$. If, in addition, $\f$ is locally Lipschitz continuous, then uniqueness follows. In the case that $\f$ satisfies the so-called Carath\'{e}odory existence/uniqueness conditions (allowing for discontinuities with respect to $t$ --- see \cite{filippov,cortes2008discontinuous}), then \eqref{eq.ODEsystem} admits a unique Carath\'{e}odory solution (i.e., an absolutely continuous solution) local to $t_0$ satisfying \eqref{eq.ODEsystem} almost everywhere. Finally, in the case that $\f$ is piecewise continuous (e.g., $PWD^{1}$), then Filippov \cite{filippov} established the following: if $\f$ is measurable and locally essentially bounded, then \eqref{eq.ODEsystem} admits a Filippov solution, i.e., an absolutely continuous solution local to $t_0$ that satisfies the following differential inclusion almost everywhere
\begin{equation}\label{eq.filippovODE}
\dot{\x} \in \bm{\mathcal{F}}(\x,t), \quad \x(t_0)=\x_0,
\end{equation}
where the Filippov set-valued map $\bm{\mathcal{F}}(\x,t)$ is defined as
$$
\bm{\mathcal{F}}(\x,t)=\bm{\mathcal{F}}[\f](\x,t) := \bigcap_{\delta > 0} \bigcap_{\mu(N)=0} \overline{\text{co}} \, \f(B_\delta(x) \setminus N,t),
$$
where $\mu(\cdot)$ denotes Lebesgue measure of a set, $\overline{\text{co}}(\cdot)$ denotes convex closure, and $B_\delta(\x)=\{\y\in\real^n: \|\y-\x\|<\delta\}$. If, in addition, $\f$ is essentially one-sided Lipschitz or satisfies a transversality-like condition, then the Filippov solution is unique. (See \cite{cortes2008discontinuous} for details and examples.)

Let $\x(t)$ be a Filippov solution of $\dot{\x}=\f(\x,t)$. Let $V:\real^n\times\real\to\real$ be a Lipschitz and a regular  (in the sense of \cite[Definition 2.4]{shevitz1994lyapunov}) function. From \cite[Theorem 2.2]{shevitz1994lyapunov}), we have
 \begin{align}
     \frac{dV}{dt}(\x(t),t)\in \dot{\Tilde{V}}(\x(t),t),\; \text{a.e.},
 \end{align}
 where $\dot{\Tilde{V}}(\x(t),t)$ is defined as
     \begin{align}\label{eq:V_general_stability}
       \dot{\Tilde{V}}(\x(t),t) &= \bigcap_{\xi\in\partial_{\rm C} V(\x(t),t)} \xi^{\rm{T}}\begin{bmatrix}
           \mathcal{F}[\f(\x(t),t)](\x(t),t) \\1
       \end{bmatrix}.
     \end{align}
Before proceeding, we recall the notions of uniform stability and practical stability (relevant to the system in \eqref{eq:average-canonical form_intro} in averaging canonical form --- see, e.g., \cite{pokhrel2023higher,moreau2000practical}). Supposing that the ODE system in \eqref{eq.ODEsystem} admits solutions on $[t_0,\infty)$ (rescaling time if necessary), we have the following.

\begin{definition}\label{def:asymptotic_stability}
An equilibrium $\x^*$ of \eqref{eq.ODEsystem} is said to be locally uniformly asymptotically stable (LUAS) if the following hold: 
\begin{itemize}
\item Uniform stability: For  every $c_1>0$, there exists $c_2>0$ such that for all $t_0\in\real$, $\x_0\in B_{c_2}(\x^*)$ implies that $\x(t;\epsilon)\in B_{c_1}(\x^*)$ for all $t\in [t_0,\infty)$.
\item Uniform boundedness: For  every $c_2>0$, there exists $c_1>0$ such that for all $t_0\in\real$, $\x_0\in B_{c_2}(\x^*)$ implies that $\x(t;\epsilon)\in B_{c_1}(\x^*)$ for all $t\in [t_0,\infty)$.
\item  Uniform attractivity: For  every $c_2>0$ and  $c_1>0$, there exists $t_f\in [0,\infty)$ such that for all $t_0\in\real$, $\x(t_0)\in B_{c_2}(\x^*)$ implies that $\x(t;\epsilon)\in B_{c_1}(\x^*)$ for all $t\in [t_0+t_f,\infty)$. 
\end{itemize}
\end{definition}
Supposing the system in \eqref{eq:average-canonical form_intro} admits solutions  on $[t_0,\infty)$, for $\epsilon \in (0,\epsilon_0]$, for some $\epsilon_0>0$, and defining the set neighborhood $B_{\delta}(D)=\{\y\in\real^n: \inf_{\x \in D}\|\y-\x\|<\delta\}$ for $D \subset X \subseteq \real^n$,
we have the following (see, e.g., \cite{durr2013lie}).
\begin{definition}\label{def:practical_stability}
A compact subset $D \subset X \subseteq \real^n$ is said to be locally practically uniformly asymptotically stable (LPUAS) with respect to \eqref{eq:average-canonical form_intro}  
if the following conditions hold:
\begin{itemize}
\item Practical uniform stability: For  every $c_1>0$, there exists $c_2>0$ and $\hat{\epsilon}\in(0,\epsilon_0]$ such that for all $t_0\in\real$ and for all $\epsilon\in(0,\hat{\epsilon})$, $\x(t_0)\in B(c_2)$ implies that $\x(t;\epsilon)\in B(c_1)$ for all $t\in [t_0,\infty)$.
\item Practical uniform boundedness: For every $c_2>0$, there exists $c_1>0$ and $\hat{\epsilon}\in(0,\epsilon_0]$ such that for all $t_0\in\real$ and for all $\epsilon\in(0,\hat{\epsilon})$, $\x_0 \in B_{c_2}(D)$ implies that $\x(t;\epsilon)\in B_{c_1}(D)$ for all $t\in [t_0,\infty)$.
\item  Practical uniform attractivity: For  every $c_2>0$ {and}  $c_1>0$, there exists $t_f\in [0,\infty)$ and $\hat{\epsilon}\in(0,\epsilon_0]$ such that for all $t_0\in\real$ and for all $\epsilon\in(0,\hat{\epsilon})$, $\x_0 \in B_{c_2}(D)$ implies that $\x(t;\epsilon)\in B_{c_1}(D)$ for all $t\in [t_0+t_f,\infty)$. 
\end{itemize}
\end{definition}

\section{Nonsmooth analysis: introducing the lexicographic Lie bracket}\label{subsec:lexicolie}
In this part, we propose a new nonsmooth Lie bracket, called the lexicographic Lie bracket (L-Lie bracket), defined as follows for L-smooth functions.

\begin{definition}\label{def:lexicographicliebracket}
{Given some open set $X \subseteq \real^n$ and L-smooth functions $\f, \g:X\to\real^m$, the lexicographic Lie bracket (L-Lie bracket) of $\f$ and $\g$ at $\x \in X$ along $\M\in\real^{n \times k}$ is defined as
\begin{equation}\label{eq.llie}
[\f,\g]_{\rm L}(\x;\M):=\J^{\rm L}\g(\x;\M) \f(\x)-\J^{\rm L}\f(\x;\M)\g(\x).
\end{equation}\label{eq.llie_identity}
The natural L-Lie bracket of $\f$ and $\g$  is defined as 
\begin{equation}
[\f,\g]_{\rm L}(\x):=\J^{\rm L}\g(\x;\I_{n}) \f(\x)-\J^{\rm L}\f(\x;\I_{n})\g(\x).
\end{equation}\label{eq.llie_identity.natural}
}
\end{definition}
\begin{remark}
Note that if $\f$ and $\g$ are $C^1$ and $\M=\I_{n}$, then \eqref{eq.llie} simplifies as
\begin{equation}\label{eq.llie.smooth}
[\f,\g]_{\rm L}(\x;\M)=\J\g(\x) \f(\x)-\J\f(\x)\g(\x)=[\f,\g](\x),
\end{equation}
i.e., the classical Lie bracket of vector fields is recovered. 

\end{remark}

\begin{example}
Consider the functions $f(x,y)=\max(x,y)$ and $g(x,y)=|x-y|$. Then, with $\M=\I_2$,
$$f'(x,y;\I_2)=\J^{\rm L}f(x,y;\I_2)={\rm \bf slmax}([x \quad 1 \quad 0],[y \quad 0 \quad 1])$$
and
\begin{align*}
g'(x,y;\I_2)
=\J^{\rm L}g(x,y;\I_2)
&=\text{fsign}(x-y,[1 \quad -1]\I_2)[1 \quad -1]\I_2
=\text{fsign}(x-y,1,-1)[1 \quad -1].
\end{align*}
Thus,
\begin{align*}
[f,g]_{\rm L}(x,y)
&=|x-y|{\rm \bf slmax}([x \quad 1 \quad 0],[y \quad 0 \quad 1])
-\max(x,y) \text{fsign}(x-y,1,-1)[1 \quad -1]\max(x,y)\\
&=\begin{cases}
[-y \quad x], &\text{if } x \geq y,\\
[y \quad -x], &\text{if } x<y.
\end{cases}
\end{align*}
See \cref{fig.Llie} for an illustration.
\begin{figure}[h!]
    \centering
    \includegraphics[width=0.7\linewidth]{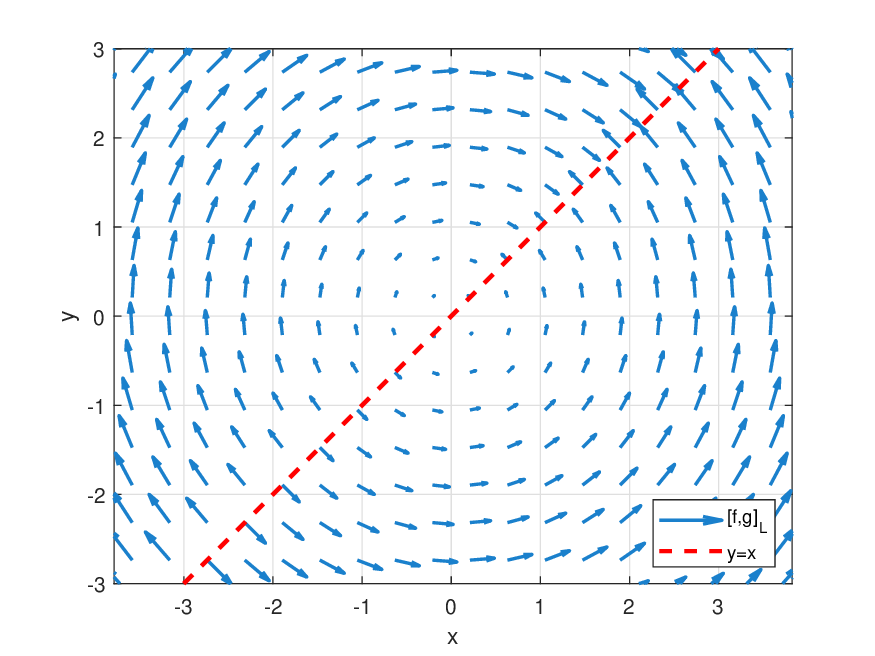}
    \caption{L-Lie bracket $[f,g]_{\rm L}(x,y)$.}
    \label{fig.Llie}
\end{figure}
If, on the other hand, $\M=-\I_2$, then
\begin{align*}
f'(x,y;-\I_2)&={\rm \bf slmax}([x,-1,0],[y,0,-1])
  \Rightarrow \; \J^{\rm L}f(x,y;-\I_2)={\rm \bf slmax}([x,-1,0],[y,0,-1])(-\I_2)
\end{align*}
and
\begin{align*}
g'(x,y;-\I_2)&=\text{fsign}(x-y,[1 \quad -1](-\I_2))[1 \quad -1](-\I_2)
=\text{fsign}(x-y,-1,1)[-1 \quad 1] \\
&\Rightarrow \; \J^{\rm L}g(x,y;-\I_2)=\text{fsign}(x-y,-1,1)[1 \quad -1].
\end{align*}
Hence,
\begin{align*}
[f,g]_{\rm L}(x,y;-\I_2)
&=|x-y|{\rm \bf slmax}([x \quad 1 \quad 0],[y \quad 0 \quad 1])-\max(x,y) \text{fsign}(x-y,1,-1)[1 \quad -1]\max(x,y)\\
&=\begin{cases}
[-y \quad x], &\text{if } x > y,\\
[y \quad -x], &\text{if } x \leq y,
\end{cases}
\end{align*}
i.e., the vector field has switched along $y=x$ from the $\M=\I_2$ case.
\end{example}

Next, {we introduce notions in a similar spirit to the class of ``piecewise continuous'' functions of, e.g., Cortes \cite{cortes2008discontinuous}  and Filippov \cite{filippov}, and the ``piecewise differentiable ($PC^1$)'' functions of Scholtes \cite{Scholtes}.}  We are motivated to do so because these classes remain broad while ensuring the  ``regularity'' of the lexicographic Lie bracket defined above under operations of interest to come in nonsmooth averaging theory.

\begin{definition}\label{def:pwsc_infinity}
{Given some open set $X \subseteq \real^n$, a function $\f:X\to\real^m$ is said to be piecewise discontinuous $C^\infty$ ($PWD^{\infty}$) if there exists a finite collection of open and connected sets $X_k$, $k=1,\ldots,q$, 
 with   $X_i \cap X_j = \emptyset$ for all $i \neq j$ and $\cup_{k=1}^{q} \text{cl}(X_k)=X$,  such that $\f$ is $C^{\infty}$ on $X_k$ for each $k$ and admits a  continuous extension to $\text{cl}(X_k)$. The function $\f$ is said to be piecewise continuous $C^\infty$ ($PWC^{\infty}$) if, in addition, $\f$ is continuous on $X$.}
\end{definition}

Note that $PWD^{\infty}$ extends the notion of ``piecewise continuity'' of, e.g. Cortes  \cite{{cortes2008discontinuous}}, and is closely related to piecewise smooth  maps with degree of smoothness one of di Bernardo et al.\  \cite{bernardo2008piecewise}, while $PWC^{\infty}$ is a subset of the $PC^1$  class of functions (as will be shown in \cref{prop:PWC_implies_L-smooth}) set forth by Scholtes \cite{Scholtes}. Moreover, compositions (and therefore sums, products, etc.) of $PWC^{\infty}$ (or $PWD^{\infty}$) functions are $PWC^{\infty}$ (or $PWD^{\infty}$), respectively.
Note also that in either case, the set of discontinuous points (or non-differentiable points) are restricted to a set of measure zero, $Z_{\f}$, contained in the boundaries of the ``pieces'', i.e., $Z_{\f} \subset \cup_{i=1}^{q} \text{bd}(X_k)$.

\begin{proposition}\label{prop:PWC_implies_L-smooth}
    Let $\f:X \to \real^m$, $X \subseteq \real^n$ open, be a $PWC^{\infty}$ function. Then, $\f$ is $PC^1$, and therefore L-smooth, on $X$.
\end{proposition}
\begin{proof}
    From \cref{def:pwsc_infinity}, there exists finite collection of open and connected sets $X_k$, $k=1,\ldots,q$ such that $\f$ is $C^{\infty}$ on $X_k$ for each $k$. Hence, there exists a finite set of functions $\f_k:X_k\to\real^m$, $k=1,\dots,q$, such that the inclusion $\f(\x)\in\{\f_1(\x),\f_2(\x),\dots,\f_q(\x)\}$ holds for every $\x\in X$, and the result follows.
\end{proof}


For use later in this paper, we make the following observations.

\begin{lemma}\label{lemma:PWSC_implies_PWC}
Let $\f:X \times \real \to \real^n$ be $PWC^{\infty}$, with $X \subseteq \real^n$ open and $Z_{\f} \subset X \times \real$  its zero-measure subset of discontinuities.
Then, for any $T>0$ and compact set $D \subset X$, the function
$$\h(\x):=
\begin{cases}
\int_0^T \frac{\partial \f}{\partial \x}(\x,t)dt, & \text{if } (\x,t) \notin Z_{\f},\\
\zero, & \text{if } (\x,t) \in Z_{\f},
\end{cases}$$
is $PWD^{\infty}$ on $D$.
\end{lemma}

\begin{proof}
As $\f$ is $PWC^{\infty}$ (with sets $X_k$), it is $C^{\infty}$ almost everywhere, i.e., for all $(\x,t) \in (X \times \real) \setminus Z_{\f}$. Hence, the partial derivative $\frac{\partial \f}{\partial \x}$ exists for almost all $(\x,t) \in (X \times \real)$, and is $C^{\infty}$ for all $(\x,t) \in (X \times \real) \setminus Z_{\f}$. In fact, $\frac{\partial \f}{\partial \x}$ is $PWD^{\infty}$ on $ X \times\real$, with the same pieces $X_k$ as admitted by $\f$. Because $\frac{\partial \f}{\partial \x}$ admits a continuous extension to $Z_{\f}$, it follows that there exists $M$ such that
$$\left|\frac{\partial \f}{\partial \x}(\x,t)\right| \leq M \quad \text{a.e. } (\x,t) \in D \times [0,T].$$
Moreover, 
by the dominated convergence theorem, we have that
$$\int_0^T \frac{\partial \f}{\partial \x}(\x,t)dt=\frac{\partial}{\partial \x} \int_0^T \f(\x,t)dt.$$
The result follows by noting that if $\f$ is $PWC^{\infty}$, then so is $\g(\x)=\int_0^T \f(\x,t)dt$, and its partial derivative will be $PWD^{\infty}$ by the same arguments as above.
\end{proof} 

\begin{lemma}\label{lemma:PWSC_implies_PWC.new}
Let $\f:X \times \real \to \real^n$ be $PWC^{\infty}$, with $X \subseteq \real^n$ open and $Z_{\f} \subset X \times \real$  its zero-measure subset of nondifferentiability. Define the following function:
\begin{equation}\label{eq:Jacobian_Like_dropping_zero}
    \widetilde{\J}_{\x}\f(\x,t):=
\begin{cases}
 \frac{\partial \f}{\partial \x}(\x,t), & \text{if } (\x,t) \notin Z_{\f},\\
\zero, & \text{if } (\x,t) \in Z_{\f}.
\end{cases} 
\end{equation}
Then, for any $\x \in X$, the function
$\h(t):=\int_0^t \widetilde{\J}_{\x}(\x,s)ds $ 
is absolutely continuous. 
\end{lemma}

\begin{proof}

Since $\f$ is $PWC^{\infty}$, 
$\frac{\partial \f}{\partial \x}$ exists a.e. and is locally bounded. 
It follows that, for fixed $\x \in X$ and $T>0$, the function $\widetilde{\J}_{\x}\f(\x,\cdot)$ is measurable and essentially bounded on $[0,T]$. Thus, $\widetilde{\J}_{\x}\f(\x,\cdot) \in L^1([0,T])$ and it follows that $\h(t)$ is absolutely continuous.
\end{proof}

\begin{lemma}\label{prop:LDerivative_PWC_PWD}
Let $\f:X \times \real\to\real^n$ be $PWC^{\infty}$, with $X \subseteq\real^n$ open. Then, for any $t\in\real$ and directions matrix $\M\in\real^{n \times k}$, the function
\begin{equation}\label{eq.h.JL}
\h(\x):=\J^{\rm L}\f(\x,t;(\M,\zero))
\end{equation}
is $PWD^{\infty}$ on $X$.
\end{lemma}

\begin{proof}
Suppose the function $\f$ admits the   $PWC^{\infty}$  construction with open and connected sets $X_k$, $k=1,\ldots,q$ such that $\f$ is $C^\infty$ on each $Z_k$. Then, observing that,
\begin{align*}
\J^{\rm L}\f(\x,t;(\M,\zero))
&=\J \f(\x,t)\begin{bsmallmatrix} \M \\ \zero \end{bsmallmatrix} \quad \text{a.e.}\\
&=\frac{\partial \f}{\partial \x}(\x,t)\M+\frac{\partial \f}{\partial t}(\x,t)\zero \quad \text{a.e.}\\
&=\frac{\partial \f}{\partial \x}(\x,t)\M  \quad \text{a.e.},
\end{align*}
i.e., $\h(\x)=\frac{\partial \f}{\partial \x}(\x,t)\M $ for $(\x,t) \in X_k$, it follows that $\h$ is $PWD^{\infty}$.
\end{proof}

We are now ready to give the main result of this section, concerning the integral of an L-Lie bracket.

\begin{proposition}\label{lemma.llie} 
Let $\f,\g:X \times \real \to \real^n$ be $PWC^{\infty}$, with  $X \subseteq\real^n$ open.  Then, for any $T>0$ and directions matrix $\M\in\real^{n \times k}$ and compact subset $D \subset X$, the function 
 \begin{align}\label{eq:lemma.llie}
 \h(\x)&:=\int_0^T [\f,\g]_{\rm L}(\x,t;(\M,\zero))dt
 \end{align}
is $PWD^{\infty}$ on $D$.  
\end{proposition}  

\begin{proof}
Observe that
\begin{align*}
[\f,\g]_{\rm L}(\x,t;(\M,\zero))
&=\J^{\rm L}\f(\x,t;(\M,\zero)) \g(\x,t)-\J^{\rm L}\g(\x,t;(\M,\zero))\f(\x,t)\\
&=\J\f(\x,t)\begin{bsmallmatrix} \M \\ \zero \end{bsmallmatrix} \g(\x,t)-\J\g(\x,t)\begin{bsmallmatrix} \M \\ \zero \end{bsmallmatrix}\f(\x,t) \quad \text{a.e.}\\
&=\frac{\partial \f}{\partial \x}(\x,t)\M \g(\x,t)-\frac{\partial \g}{\partial \x}(\x,t)\M \f(\x,t) \quad \text{a.e.}
\end{align*}
Hence,
$$\h(\x)=\int_0^{T} \left(\frac{\partial \f}{\partial \x}(\x,t) \M \g(\x,t)-\frac{\partial \g}{\partial \x}(\x,t)\M \f(\x,t)\right)dt$$
almost everywhere. The result then follows from the fact that products of $PWD^{\infty}$  and $PWC^{\infty}$ functions are $PWD^{\infty}$ functions and  \cref{lemma:PWSC_implies_PWC}.
\end{proof}
Recalling the LD-derivative notation established earlier, we note that $\M=\I$ (with appropriate dimensions) will be the case going forward in the paper. Hence, as a result of \eqref{eq.LD.L.derivative}, the LD-derivatives calculated below will be equal to L-derivatives. 
\section{{Nonsmooth averaging theory}}\label{sec:HigherOrderAvg}
 In this part, we establish first- and second-order averaging theory for the NLTP system written in the averaging-canonical form in \eqref{eq:average-canonical form_intro}. This is accomplished using a nonsmooth near-identity transformation.
\subsection{Nonsmooth near-identity transformations and first-order averaging}
Analogous to the smooth near-identity (classical) transformation \cite{sanders2007averaging,maggia2020higher}, we seek  a transformation that converts the states, represented by $\x$, in \eqref{eq:average-canonical form_intro} into an averaging system, represented by $\y$, whose associated governing equations can be made time-invariant. To do so, we  introduce  the following so-called Lipschitz continuous {near-identity} transformation: 
\begin{align}
\label{eq:Nonsmooth_x_y_transformation}
\x=\capu(\y,t,\epsilon)=\y+\epsilon\w(\y,t,\epsilon),
\end{align}
for some function $\w$ that is 
locally Lipschitz continuous in $\y$ and $t$, and analytic in $\epsilon$. 
 We prove the invertibility (and hence usefulness) of  the transformation
 via a similar rationale to the smooth near-identity transformation \cite[Section 2.9]{sanders2007averaging}, {with the classical implicit function theorem usage replaced by a nonsmooth implicit function theorem}.
 
\begin{theorem}\label{thm:NonsmoothTransformationValidity}(\textbf{Near-Identity Transformation}) 
{Consider the Lipschitz continuous transformation in} \eqref{eq:Nonsmooth_x_y_transformation}.
{Then, for} any bounded connected open set $X\subset\real^n$, there exists $\epsilon_0>0$  such that for all $t\in\real$ and for all $\epsilon \in [0,\epsilon_0]$, {the mapping $\y\mapsto \capu(\y, t, \epsilon)$ carries $X$}   one-to-one  and onto its image $\capu(X, t, \epsilon)$,  with the following inverse mapping:
\begin{align}\label{eq:inverse_mapping_x_y}
    \y=\capv(\x, t, \epsilon) = \x+\epsilon \vv(\x, t, \epsilon)
\end{align}
for some function $\vv$, such that $\capv$ is  {locally}  Lipschitz continuous in $\x$ and $t$, and analytic in $\epsilon$.
\end{theorem}

\begin{proof}
The proof to show that $\capu$ is one-to-one on $X$ (and hence maps $X$ invertibly onto $\capu(X,t,\epsilon)$) for any $t$ and small enough $\epsilon$  follows identically to the proof  of \cite[Lemma 2.8.3]{sanders2007averaging}, since $\w$ is locally Lipschitz with respect to $\y$.
Next we show  Lipschitz continuity of the transformation and the form of the inverse. 
Let $\GG(\x,\y,t,\epsilon) = -\x + \y + \epsilon \w(\y,t,\epsilon)$.
Let $\hat{\epsilon}=0$ and let $\hat{\x}\in X$, $\hat{t}\in\real$. Let $\hat{\y}=\hat{\x}$. Then $\GG(\hat{\x},\hat{\y},\hat{t},\hat{\epsilon})
=\zero_n$. Then, since 
$\pi_{\y} \partial_{\rm C}\GG(\hat{\x},\hat{\y},\hat{t},\hat{\epsilon}) = \pi_{\y} \partial_{\rm C}\GG(\hat{\x},\hat{\y},\hat{t},0)=\{\I\}$ 
it follows that  $\pi_{\y} \partial_{\rm C} \GG(\hat{\x}, \hat{\y},t,0)$ is of maximal rank {(i.e., it contains no singular matrices)}.
Then, we call the nonsmooth implicit function theorem in \cite[Subsection 7.1]{clarke1990optimization} to conclude  that for small enough {$(\x,t,\epsilon)$} 
(i.e. in a neighborhood $N:=N_{\hat{\x}}\times N_{\hat{t}}\times N_{\hat{\epsilon}}$ {of $(\hat{\x},\hat{t},\hat{\epsilon})$}, where $N_{\hat{\x}},N_{\hat{t}}$ and $ N_{\hat{\epsilon}}$ are neighborhoods around $\hat{\x},\hat{t}$ and $\hat{\epsilon}$, respectively), there exists a  Lipschitz continuous function $\rr:N\to\real^n$ such that, for each ${(\x,t,\epsilon) \in N}$, ${(\x,\rr(\x,t ,\epsilon),t,\epsilon)}$ is the unique vector in a neighborhood of $(\hat{\x},\hat{\y},\hat{t},\hat{\epsilon})$ satisfying $\GG(\x,\rr(\x,t ,\epsilon),t,\epsilon)=\zero_n$. {Moreover, as argued in \cite[Lemma 2.8.3]{sanders2007averaging}, the local inverse function of $\y \mapsto \capu(\y,t,\epsilon)$ near $\hat{\y}$, which equals the implicit function $\y=\rr(\x,t,\epsilon)$ locally, takes the form of \eqref{eq:inverse_mapping_x_y}.}

{This process can be repeated for finitely many points $(\hat{\x}_i,\hat{t},\hat{\epsilon})$ to construct finitely many  local functions  $\rr_i:N_i$ (where $N_i$, and thus $N_{i,\hat{\epsilon}}$, depends on $\hat{\x}_i$) such that $\cup_i N_{i,\hat{\x}_i} \supset \text{cl}(X)$, since $\text{cl}(X)$ is compact. The global implicit function (which coincides with the global inverse function of $\y \mapsto \capu(\y,t,\epsilon)$) can then be constructed for some $t\in\real$ and $0 \leq \epsilon \leq \epsilon_0$, with $\epsilon_0>0$ the minimum of: (i) $1/L_{\w}$ (with $L_{\w}>0$ the Lipschitz constant associated with $\w$ on \text{cl}(X)), and  (ii) the sizes of the neighborhoods $N_{i,\hat{\epsilon}_i}$ (see \cite[Lemma 2.8.3]{sanders2007averaging} for details), to get:
$$\capv:(\x,t,\epsilon) \mapsto \rr_i(\x,t,\epsilon) \text{ if } \x \in N_{i,\hat{\x}_i}.$$
}
Hence, the inverse of  $\x=\capu(\y, t, \epsilon)$ is locally  Lipschitz continuous in $\x$, and {takes the form of  \eqref{eq:inverse_mapping_x_y}}.
\end{proof}

   \cref{thm:NonsmoothTransformationValidity} establishes a nonsmooth near-identity transformation that can be used for averaging the NLTP in \eqref{eq:average-canonical form_intro}. In particular, the auxiliary variable $\y=\x+\epsilon \vv(\x,t,\epsilon)$ transforms the NLTP into the averaging system.
    The goal then is to choose $\vv$ (or $\w$) so that {the differential equation associated with} $\y$ can be made time-invariant --- in such a case, it automatically represents the nonsmooth complete averaging of $\x$, i.e., the solution of the NLTP in \eqref{eq:average-canonical form_intro}. 


The following base assumptions on the RHS function $\f:X\times \real\times \real_+ \to X$ are needed in the results to come:
\begin{itemize}
    \item[\textbf{(A1)}] $\f(\x,t,\epsilon)$ is L-smooth in $\x$.
    \item[\textbf{(A2)}] $\f(\x,t,\epsilon)$ is locally Lipschitz continuous and $T$-periodic in $t$.
    \item[\textbf{(A3)}] $\f(\x,t,\epsilon)$ is analytic in $\epsilon$. 
\end{itemize}
Note that the solution $\x(t)=\x(t;\epsilon)$ of the NLTP system \eqref{eq:average-canonical form_intro} under assumptions (A1)-(A3) is guaranteed to be $C^1$ (and hence, Lipschitz continuous)
    in $t$ 
    and analytic in $\epsilon$. Moreover, from assumptions (A1)-(A3), we can expand the RHS of \eqref{eq:average-canonical form_intro} as a power series of $\epsilon$ (about $\epsilon =0$):
\begin{align}\label{eq:xExpansion2ndorder}
     \dot{\x} = \epsilon \f_1(\x,t)+\frac{\epsilon^2}{2!}\f_2(\x,t) +\frac{\epsilon^3}{3!}\hat{\f}(\x,t,\epsilon),
 \end{align}
where  
\begin{align*}
\f_1(\x,t)&:=\f(\x,t,0),\quad
\f_2(\x,t):= 2 \frac{\partial \f }{\partial \epsilon}(\x,t,0),
\end{align*}
and $\hat{\f}(\x,t,\epsilon)$
denotes the remainder term\footnote{Note that the coefficient $2$ appears in the expression $\f_2(\x,t)= 2 \frac{\partial \f }{\partial \epsilon}(\x,t,0)$ since there is an extra $\epsilon$ multiplied by $\f(\x,t,\epsilon)$ in \eqref{eq:average-canonical form_intro}, we multiplied $\frac{\partial \f }{\partial \epsilon}(\x,t,0)$ by 2 to match the coefficients in \eqref{eq:xExpansion2ndorder} with the factorials in traditional Taylor expansions.}. Note that $\f_1$ and $\f_2$ inherit periodicity from $\f$. 

\begin{proposition}\label{lemma:f_1_f_2_peridic}
Suppose  that assumption (A2) holds. Then $\f_1(\x,t)$ and $\f_2(\x,t)$ in \eqref{eq:xExpansion2ndorder} are $T-$periodic in $t$.
 \end{proposition}
 
 \begin{proof}
     Since $\f_1(\x,t)=\f(\x,t,0)$, it is clear that:
     \begin{align*}
\f_1(\x,t)=\f(\x,t,0) = \f(\x,t+T,0)=\f_1(\x,t+T).
     \end{align*}
     Also, since $\f(\x,t,\epsilon) = \f(\x,t+T,\epsilon)$, then $\frac{\partial \f}{\partial \epsilon}(\x,t,\epsilon) =\frac{\partial \f}{\partial \epsilon}(\x,t+T,\epsilon)$. It follows that:
          \begin{align*}
         \f_2(\x,t) = 2 \frac{\partial \f }{\partial \epsilon}(\x,t,0)= 2 \frac{\partial \f }{\partial \epsilon}(\x,t+T,0)=\f_2(\x,t+T).
     \end{align*}
 \end{proof}

In working towards the averaging of \eqref{eq:average-canonical form_intro}, we first construct an auxiliary time-dependent system via the near-identity transformation.

\begin{lemma}\label{lemma:g_1_g_2_hatg_Loc_Lip} 
Suppose that assumptions (A1)-(A3) hold. 
Let $\x(t)=\x(t;\epsilon)$ be the solution of \eqref{eq:average-canonical form_intro} for $\epsilon \in [0,\epsilon_0]$, for some $\epsilon_0>0$. 
 Let $\y(t)=\y(t;\epsilon)$ be defined according to 
 \begin{equation}\label{eq.lemmag1g2}
 \y(t):=\capv(\x(t), t, \epsilon)= \x(t)+\epsilon \vv(\x(t), t, \epsilon),
 \end{equation}
for some  function $\vv$ that is locally Lipschitz continuous in $\x$, $t$, and analytic in $\epsilon$. Then it follows that 
$\y(t)$ is a Filippov
solution of 
\begin{equation}\label{eq:yExpansion2ndorder}
     \dot{\y} = \epsilon \h_1(\y,t)+\epsilon^2\hat{\h}(\y,t,\epsilon),
 \end{equation}
for some functions $\h_1(\y,t)$ and $\hat{\h}(\y,t,\epsilon)$ that are measurable and essentially bounded in $\y, t$ and analytic in $\epsilon$.
\end{lemma}

\begin{proof}
As the solution $\x$ is $C^1$ in $t$ and $\capv$ is locally Lipschitz with respect to $\x$ and $t$, it follows that $\y$ is locally Lipschitz (and thus absolutely continuous). Hence,  it follows from \eqref{eq:average-canonical form_intro}  that, {for almost all $t$},
    \begin{align*}
        \dot{\y}
        = \dot{\x} + \epsilon \dot{\vv}(\x,t,\epsilon)
        &=\dot{\x} + \epsilon\left(\frac{\partial \vv}{\partial\x}(\x,t,\epsilon)\dot{\x}+\frac{\partial\vv}{\partial t}(\x,t,\epsilon)\right)\\
        &=\epsilon\left(\f(\x,t,\epsilon)+\epsilon\frac{\partial \vv}{\partial\x}(\x,t,\epsilon)\f(\x(,t,\epsilon)+\frac{\partial\vv}{\partial t}(\x,t,\epsilon)\right).
    \end{align*}
 Note that $\f$ and ${\vv}$ are analytic with respect to $\epsilon$, and $\f$ and $\vv$ are  locally Lipschitz in $\x$ and $t$  (and hence $\frac{\partial \vv}{\partial \y}$ and $\frac{\partial \vv}{\partial t}$ exist a.e. and are measurable and essentially bounded). Thus,
    \begin{equation}\label{eq.gyte}
    \h^*(\x,t,\epsilon):=\f(\x,t,\epsilon) + \epsilon\frac{\partial \vv}{\partial\x}(\x,t,\epsilon)\f(\x,t,\epsilon)+\frac{\partial\vv}{\partial t}(\x,t,\epsilon) 
    \end{equation}
    exists a.e., is measurable and essentially bounded in $\x$ and $t$, and is analytic with respect to $\epsilon$. Let 
    $$\pmb{\gamma}(\y,t,\epsilon):=(\y+\epsilon \w(\y,t,\epsilon),t,\epsilon)=(\x,t,\epsilon)$$ 
    and define
    $\h(\y,t,\epsilon):=\h^*(\pmb{\gamma}(\y,t,\epsilon))$.
    We see that $\h$ exists a.e., is measurable and essentially bounded in $\y$ and $t$, and is analytic with respect to $\epsilon$.    
    Expanding $\h(\y,t,\epsilon)$ in terms of $\epsilon$ about $\epsilon=0$, we get
    \begin{align*}
        \dot{\y}=\epsilon\h^*(\pmb{\gamma}(\y,t,\epsilon)) =\epsilon\h(\y,t,\epsilon)&= \epsilon \h_1(\y,t)+\epsilon^2\hat{\h}(\y,t,\epsilon),
    \end{align*}
    where $\h_1$ and $\hat{\h}$ are the Taylor expansion terms of $\h(\y,t,\epsilon)$ when expanding about $\epsilon=0$, with smoothness properties of $\h_1$ and $\hat{\h}$ inherited from $\h$, and the result follows.
\end{proof}

As expected, the original variable $\x$, and the transformed variable $\y$, remain ``close''.

\begin{corollary}\label{lemma:g_1_g_2_hatg_Loc_Lip.corollary} 
Assume the setting of \cref{lemma:g_1_g_2_hatg_Loc_Lip}. Then it holds that
\begin{align}\label{eq:x_y_bounds}
    \|\x(t;\epsilon)-\y(t;\epsilon)\|=\mathcal{O}(\epsilon){.}
\end{align}
\end{corollary}

\begin{proof}
This follows immediately from \eqref{eq.lemmag1g2}.
\end{proof}

The function $\h_1$ in \eqref{eq:yExpansion2ndorder} is time-dependent. However, as will be demonstrated below, it is possible to choose $\vv$ (or $\w$ rather), and hence $\capv$ (or $\capu$ rather), in a way so that the right-hand side functions of the auxiliary equation become time-invariant, by ``shifting'' the time dependency in $\h_1$ to $\hat{\h}$. This is accomplished by the typical approach of averaging $\f$ over a period to capture its oscillatory behavior.  In particular,  we introduce the following functions related to  averaging: \looseness=-1
\begin{align}
\g_1(\y)&:=\Bar{\f}_1(\y)=\frac{1}{T}\int_{0}^{T}\f_1(\y,t)dt,\label{eq.g1}\\
\w(\y,t)&:= \int_0^t(\f_1(\y,s)-\Bar{\f}_1(\y))ds.\label{eq.w1}
\end{align}
Analogous to smooth averaging theory, $\g_1$ is a time-invariant functions that capture the oscillatory behavior of the NLTP system. The function $\w$, whose T-periodicity follows from \cref{lemma:f_1_f_2_peridic}, ``captures'' the difference. 

  \begin{proposition}\label{prop.g1w1}
Suppose that assumptions (A1)-(A3) hold. Then the functions $\g_1(\y)$ and $\w(\y,t)$ are L-smooth. Additionally, $\w(\y,t)$ is T-periodic in $t$, and $\w(\y,T)=\zero$ for all $\y$. 
 \end{proposition}
 
\begin{proof}
First, we note that $\f_1$ is L-smooth, by assumption (A1), and we recall that $\f_1$ is T-periodic in $t$ by \cref{lemma:f_1_f_2_peridic}.
Since $\f_1$ is L-smooth, it follows immediately that $\g_1$ is L-smooth. From this,   
$$\w(\y,t)=\int_0^t \f_1(\y,s)ds-\int_0^t\bar{\f}_1(\y)ds=\int_0^t \f_1(\y,s)ds- t \g_1(\y)$$
is also L-smooth, as the difference of L-smooth functions.

Lastly, observe that
$$\w(\y,T)=\int_0^T \f_1(\y,s)ds-\int_0^T \Bar{\f}_1(\y)ds=T\Bar{\f}_1(\y)-T\Bar{\f}_1(\y)=\zero$$
and so, by periodicity of $\f_1$,
\begin{align*} 
\w(\y,t+T)
&= 
\int_0^{t}(\f_1(\y,s)-\Bar{\f}_1(\y))ds+\int_t^{t+T}(\f_1(\y,s)-\Bar{\f}_1(\y))ds\\
&= 
\int_0^{t}(\f_1(\y,s)-\Bar{\f}_1(\y))ds+\int_0^{T}(\f_1(\y,s)-\Bar{\f}_1(\y))ds
=\w(\y,t)+\w(\y,T)
=\w(\y,t).
\end{align*}
\end{proof}
With the above in place, we show that the auxiliary system in \eqref{eq:yExpansion2ndorder}, which we can consider to be the  NLTP under the near-identity transformation, can essentially be rewritten with $\h_1$ as a time-independent function (not necessarily uniquely so).

\begin{theorem}\label{thm:y_bar_y_bounds}
Suppose that assumptions (A1)-(A3) hold. 
 Let $\x(t)=\x(t;\epsilon)$ be the solution of \eqref{eq:average-canonical form_intro} for $\epsilon \in [0,\epsilon_0]$, for some $\epsilon_0>0$, such that $\x(t) \in D$ for some compact set $D \subset X \subseteq \real^n$, and $t\in T:=[0,\frac{\kappa}{\epsilon}]$, for some  $\kappa>0$.
Let $\y(t)=\y(t;\epsilon)$ be defined according to $\y(t):=\capv(\x(t), t, \epsilon)= \x(t)+\epsilon \vv(\x(t), t, \epsilon)$, and
let $\Bar{\y}(t)=\Bar{\y}(t;\epsilon)$ be a solution of 
\begin{equation} \label{eq:averaged_sys}  
    \dot{\Bar{\y}} = \epsilon \g_1(\Bar{\y}) + \ \epsilon^2 \bgHat(\Bar{\y}, t, \epsilon), 
\end{equation}
where   $\bgHat$ is measurable and essentially bounded. Then 
\begin{align}
    \| \y(t;\epsilon) - \Bar{\y}(t;\epsilon) \| &= \mathcal{O}(\epsilon) \label{eq:y_bar_y_est}
\end{align}
for $t \in T$, i.e., for $t = \mathcal{O}(1/\epsilon)$.
\end{theorem}

\begin{proof}
Differentiating the transformation \eqref{eq:Nonsmooth_x_y_transformation} with respect to time yields the following a.e.:
\begin{equation*}
    \dot{\x} = \dot{\y} + \epsilon \frac{\partial\w}{\partial \y}\dot{\y} + \epsilon \frac{\partial{\w}}{\partial t} = \left( \I_{n} + \epsilon  \frac{\partial\w}{\partial \y}\right) \dot{\y} + \epsilon \frac{\partial{\w}}{\partial t}.
\end{equation*}
First, we argue that the linear operator $\epsilon \frac{\partial{\w}}{\partial \y}$ is a bounded linear operator with spectral radius less than one, i.e., $\rho(\epsilon \frac{\partial{\w}}{\partial \y})<1$. This is because of the fact that $\w$ is locally Lipschitz continuous with respect to $\y$ by some Lipschitz constant $L$, and therefore, $\|\frac{\partial{\w}}{\partial \y}\|\le L$ a.e. on $D\times T$. Hence, we see that $\rho(\epsilon \frac{\partial{\w}}{\partial \y})\le\|\epsilon \frac{\partial{\w}}{\partial \y}\|\le |\epsilon|\| \frac{\partial{\w}}{\partial \y}\|<1$, provided that $\epsilon<\frac{1}{L}$.  Hence, the Neumann series $\Sigma_{k=0}^{\infty}(\epsilon \frac{\partial{\w}}{\partial \y})^k$ converges, and we can use the expansion $(I + \epsilon \frac{\partial{\w}}{\partial \y})^{-1} = \I_{n} - \epsilon \frac{\partial{\w}}{\partial \y} + \mathcal{O}(\epsilon^2)$. Rearranging for $\dot{\y}$, we get:
\begin{equation*}
    \dot{\y}  = \left( \I_{n} - \epsilon\frac{\partial\w}{\partial \y}(\y,t) \right) \left( \dot{\x}(t)  - \epsilon \frac{\partial{\w}}{\partial t}(\y,t) \right) + \mathcal{O}(\epsilon^3)
\end{equation*}
Substituting the expansion of $\dot{\x}$ in \eqref{eq:xExpansion2ndorder} and using the homological equation $\frac{\partial{\w}}{\partial t} = \f_1(\y, t) - \g_1(\y)$, we get a.e.:
\begin{equation}\label{eq:dot_y_dot_x}
    \dot{\y}  = \left( \I_{n} - \epsilon \frac{\partial\w}{\partial \y}(\y,t) \right) \left( \epsilon \f_1(\x, t)  - \epsilon (\f_1(\y, t) - \g_1(\y)) \right) + \mathcal{O}(\epsilon^2)
\end{equation}
Expanding $\f_1(\x, t) = \f_1(\y + \epsilon \w(\y,t), t)$ around $\y$ using \cref{prop:partial_L_Taylor}, we get:
\begin{align*}
\f_1(\x,t)&=\f_1(\y,t)+\JL\f_1(\y,t;\begin{bsmallmatrix}
            \epsilon \w(\y, t)& \I_n & \zero_{n}\\ 0&\zero_{1\times n}&1
        \end{bsmallmatrix})\begin{bsmallmatrix}
            \epsilon \w(\y, t)\\0
        \end{bsmallmatrix} + \mathcal{O}(\|
            \epsilon \w(\y, t)\|^2)\\
    &= \f_1(\y, t) + \epsilon \JL_{\y} \f_1 (\y,t) \w(\y, t) + \mathcal{O}(\epsilon^2).
\end{align*}
Substituting back into \eqref{eq:dot_y_dot_x}, the $\mathcal{O}(\epsilon)$ terms cancel, and we get the following a.e.:
\begin{align*}
    \dot{\y}  &= \left( \I_{n} - \epsilon \frac{\partial\w}{\partial \y}(\y,t) \right) \Bigg( \epsilon(  \f_1(\y, t) + \epsilon \JL_{\y}\f_1 (\y,t) \w(\y, t) + \mathcal{O}(\epsilon^2))
    - \epsilon (\f_1(\y, t) - \g_1(\y)) \Bigg) + \mathcal{O}(\epsilon^2)\\
    &=\left( \I_{n} - \epsilon \frac{\partial\w}{\partial \y}(\y,t) \right) \Bigg( \epsilon(  \epsilon \JL_{\y}\f_1 (\y,t) \w(\y, t) + \mathcal{O}(\epsilon^2))
     + \epsilon \g_1(\y) \Bigg) + \mathcal{O}(\epsilon^2)\\
    &=\left( \I_{n} - \epsilon \frac{\partial\w}{\partial \y}(\y,t) \right)  \epsilon \g_1(\y)  + \mathcal{O}(\epsilon^2)
    =\epsilon \g_1(\y) + \mathcal{O}(\epsilon^2).
\end{align*}
Hence, the function $\y(t;\epsilon)$ satisfies the following differential equation:
\begin{equation} \label{eq:y_dynamics}
    \dot{\y}  = \epsilon \g_1(\y) +  \epsilon^2 \h^\dagger(\y, t, \epsilon).
\end{equation}

Subtracting the full averaged system \eqref{eq:averaged_sys} from the transformed system \eqref{eq:y_dynamics}:
\begin{equation*}
    \dot{\y}\ -\dot{\Bar{\y}}\ = \epsilon (\g_1(\y) - \g_1(\Bar{\y})) + \epsilon^2 (\h^\dagger(\y, t, \epsilon) - \bgHat(\Bar{\y}, t, \epsilon)).
\end{equation*}
Let $\E(t) = \E(t;\epsilon) = \y(t;\epsilon) - \Bar{\y}(t;\epsilon)$. Then,
\begin{align*}
    \E(t) &= \int_0^t \epsilon (\g_1(\y(\tau)) - \g_1(\Bar{\y}(\tau))) \, d\tau   
    + \int_0^t \epsilon^2 (\h^\dagger(\y (\tau), \tau, \epsilon )- \bgHat ( \Bar{\y}(\tau), \tau, \epsilon )) \, d\tau. 
\end{align*}
Since $\g_1$ is L-smooth,  it is locally Lipschitz continuous, i.e.,   there exists a constant $L_{\g_1}>0$ such that $\|\g_1(\y)-\g_1(\Bar{\y})\|\le L_{\g_1}\|\y-\Bar{\y}\|$, {for all $\y$ and $\Bar{\y}$ in $D$}.
Next, for $\epsilon \in (0,\epsilon_0]$, we define 
\begin{align*}
M_{\hat{\g}}
=\sup \left\{\hat{\g}(\y,t,\epsilon):\;\y \in D,\; t\in\left[0,\frac{\kappa}{\epsilon}\right]\right\}, \quad 
M_{\h^\dagger}&=\sup \left\{\h^\dagger(\y,t,\epsilon):\;\y \in D, \;t\in\left[0,\frac{\kappa}{\epsilon}\right]\right\}.
\end{align*}
Then, using Minkowski's integral inequality, we see that
    \begin{align*}
        \|\E(t)\|&\le \int_0^t \epsilon \|\g_1(\y(\tau)) - \g_1(\Bar{\y}(\tau))\|d\tau 
        + \epsilon^2 M_{\hat{\g}}t+ \epsilon^2 M_{\h^\dagger}t\\
        &\le \int_0^t \epsilon L_{\g_1}\|\E(\tau)\|d\tau + \epsilon^2 M_{\hat{\g}}t+ \epsilon^2 M_{\h^\dagger}t
        = (\epsilon^2 M_{\hat{\g}}+\epsilon^2 M_{\h^\dagger})t + \epsilon L_{\g_1}\int_0^t \|\E(\tau)\|d\tau.
        \end{align*}
        By using the Specific Gronwall Lemma \cite[Lemma 1.3.3]{sanders2007averaging} we get
    \begin{align*}
        \|\E(t)\|&\le   \frac{(\epsilon^2 M_{\hat{\g}}+\epsilon^2 M_{\h^\dagger})}{\epsilon} \frac{1}{L_{\g_1}} \Big(e^{\epsilon L_{\g_1}t}-1\Big)=\epsilon\frac{M_{\hat{\g}}+M_{\h^\dagger}}{L_{\g_1}}(e^{\epsilon {L_{\g_1}t}}-1),
    \end{align*}
    for all $t\in T$, i.e.,  we get $\|\E(t;\epsilon)\|=\|\y(t;\epsilon)-\Bar{\y}(t;\epsilon)\|=\mathcal{O}(\epsilon)$, as required.
\end{proof}

Note that we refer to \eqref{eq:averaged_sys} as the ``full averaged system''. Now we are in a position to formally define the nonsmooth (Lipschitzian) first-order averaging system.
 
\begin{theorem}\label{thm:first_order_averaging}
Suppose that assumptions (A1)-(A3) hold. 
Let $\x(t)=\x(t;\epsilon)$ be the solution of \eqref{eq:average-canonical form_intro} for $\epsilon \in [0,\epsilon_0]$, for some $\epsilon_0>0$, and let $\Bar{\y}(t)=\Bar{\y}(t;\epsilon)$ be a solution of \eqref{eq:averaged_sys}. Then  the nonsmooth first-order averaging system
    \begin{align}\label{eq:yExpansion1storder_truncated}
     \dot{\zz}_1 = \epsilon \Bar{\f}_1(\zz_1)
 \end{align} 
admits a (classical) solution $\zz_1(t)=\zz_1(t;\epsilon)$, which satisfies
\begin{equation}\label{eq:y_bar_y_only_g1}
       \|\Bar{\y}(t;\epsilon)-\zz_1(t;\epsilon)\|=\mathcal{O}(\epsilon),
    \end{equation}
and
\begin{align}\label{eq:x_bar_y_only_g1}
   \|\x(t;\epsilon)-\zz_1(t;\epsilon)\|=\mathcal{O}(\epsilon),
\end{align}
for $t=\mathcal{O}(1/\epsilon)$.
\end{theorem}
\begin{proof}
    The proof follows similar steps  as the second part of the proof of \cref{thm:y_bar_y_bounds}, i.e., proving the relation \eqref{eq:y_bar_y_est}. 
\end{proof}
\begin{proof}
First, recall that $\g_1(\zz_1)=\Bar{\f}_1(\zz_1)$. Then, letting $\E(t)=\E(t;\epsilon)=\Bar{\y}(t;\epsilon)-\zz_1(t;\epsilon)$, we see that 
  \begin{align*}
        \|\E(t)\| &= \|\Bar{\y}(t)-\zz_1(t)\|
        = \Big|\Big|\int_0^t \epsilon (\g_1(\Bar{\y}(\tau)) - \g_1(\zz_1(\tau)))  +\epsilon^2\hat{\g}(\Bar{\y}(\tau),\tau,\epsilon)) d\tau\Big|\Big|.
    \end{align*}
 Using Minkowski's integral inequality, we get
    \begin{align*}
        \|\E(t)\|&\le \int_0^t \epsilon \|\g_1(\Bar{\y}(\tau)) - \g_1(\zz_1(\tau))\|d\tau +  \epsilon^2 M_{\hat{\g}}t\\
        &\le \int_0^t \epsilon \|\g_1(\Bar{\y}(\tau)) - \g_1(\zz_1(\tau))\|d\tau 
         + \epsilon^2 M_{\hat{\g}}t\\
        &\le \int_0^t \epsilon L_{\g_1}\|\E(\tau)\|d\tau + \epsilon^2 M_{\hat{\g}}t
        = \epsilon^2 M_{\hat{\g}}t + \epsilon L_{\g_1}\int_0^t \|\E(\tau)\|d\tau.
        \end{align*}
        By using the Specific Gronwall Lemma \cite[Lemma 1.3.3]{sanders2007averaging} we get
    \begin{align*}
        \|\E(t)\|&\le   \frac{\epsilon^2 M_{\hat{\g}}}{\epsilon} \frac{1}{L_{\g_1}} e^{\epsilon L_{\g_1}t}  - \frac{\epsilon^2 M_{\hat{\g}}}{\epsilon} \frac{1}{L_{\g_1}}.
    \end{align*}
    As $\epsilon\rightarrow 0$, we get
    \begin{align*}
         \|\E(t)\|\le  \frac{\epsilon M_{\hat{\g}}}{L_{\g_1}}(e^{\epsilon L_{\g_1}t}-1).
    \end{align*}
    With $t=\mathcal{O}(1/\epsilon)$, we get \begin{equation}\label{eq:y_bar_y_only_g1_inproof}
        \|\E(t;\epsilon)\|=\|\Bar{\y}(t;\epsilon)-\zz_1(t;\epsilon)\|=\mathcal{O}(\epsilon).
    \end{equation}
Therefore, from \eqref{eq:x_y_bounds}, \eqref{eq:y_bar_y_est} and \eqref{eq:y_bar_y_only_g1_inproof}, we see that for $t=\mathcal{O}(1/\epsilon)$, we get:
\begin{align*}
    \|\x(t;\epsilon)-\zz_1(t;\epsilon)\|&\le \|\x(t;\epsilon)-\y(t;\epsilon)\|+\|\y(t;\epsilon) -\Bar{\y}(t;\epsilon)\|+\|\Bar{\y}(t;\epsilon)-\zz_1(t;\epsilon)\|\\
    &=\mathcal{O}(\epsilon).
    \end{align*}
\end{proof}
\begin{corollary}\label{cor:1st_order_bound}
Assume the setting of \cref{thm:first_order_averaging}. 
Let $\y(t)=\y(t;\epsilon)$ be the solution of \eqref{eq:yExpansion2ndorder} for $\epsilon \in [0,\epsilon_0]$, for some $\epsilon_0>0$.
Then 
\begin{align*}
   \|\y(t;\epsilon)-\zz_1(t;\epsilon)\|=\mathcal{O}(\epsilon),\quad \text{for } t=\mathcal{O}(1/\epsilon).
\end{align*}
\end{corollary}\label{lemma:equivalenceOftermsWhen_g1=0}
\begin{proof}
    We see from \cref{thm:y_bar_y_bounds} and \cref{thm:first_order_averaging} that we have
    \begin{align*}
        \|\y(t;\epsilon)-\zz_1(t;\epsilon)\|&\le\|\y(t;\epsilon)-\Bar{\y}(t;\epsilon)\|+\|\Bar{\y}(t;\epsilon)-\zz_1(t;\epsilon)\|=\mathcal{O}(\epsilon),\quad \text{for }t=\mathcal{O}(1/\epsilon).
    \end{align*}
\end{proof}

  \cref{thm:first_order_averaging} provides the first-order bound that is needed to prove that the NLTP system and its nonsmooth  first-order averaging system ``share'' stability properties, as seen in the next result,  which concerns practical stability.
 
\begin{theorem}\label{thm:asymptotic_implies_practical}
Suppose that assumptions (A1)-(A3) hold. 
 Let $\x^*\in X \subseteq \real^n$ be locally uniformly asymptotically stable 
 for system \eqref{eq:yExpansion1storder_truncated}. Then there exists $\epsilon>0$ such that $\x^*$ is locally practically uniformly asymptotically stable for system \eqref{eq:average-canonical form_intro}.
\end{theorem}

\begin{proof}
We see from \cref{thm:first_order_averaging} that for any $\epsilon>0$, there exists $t_f=\mathcal{O}(1/\epsilon)$ such that for all $t\in[0,t_f]$, we have $\|\x(t;\epsilon)-\zz_1(t;\epsilon)\|=\mathcal{O}(\epsilon)$, i.e., the error bound between $\x$ and $\zz_1$ can be made arbitrarily small for any positive finite period of time. Hence, we see 
that the result follows.
\end{proof}

\subsection{Nonsmooth second-order averaging theory}

 Now we shift our focus to formally define the nonsmooth (L-smooth) second-order averaging system and remark on its computation. 
 With an aim of deriving a second-order averaging systems, we 
use  lexicographic differentiation, which necessitates the strengthening of assumption (A1) from local Lipschitz continuity to piecewise continuous $C^{\infty}$ (while recalling that $PWC^{\infty}$ covers a broad range of functions, such as $\min$, $\max$, abs, etc., and implies lexicographic smoothness via {\cref{prop:PWC_implies_L-smooth}}):
 \begin{itemize}
      \item[\textbf{(A4)}] 
         $\f(\x,t,\epsilon)$ is $PWC^{\infty}$ in $\x$.
 \end{itemize}

For second-order averaging, we introduce the following function:
\begin{align}
     \g_2(\y) 
     &:= \frac{1}{T}\int_{0}^{T}(\f_2(\y,t)+[\w,\f_1]_{\rm L}(\y,t)+[\w,\g_1]_{\rm L}(\y,t)) dt.\label{eq.g2}
\end{align} 
where we use the natural L-Lie bracket with respect to $\y$, i.e.,
\begin{align*}
[\w,\f_1]_{\rm L}(\y,t)
&:=[\w(\cdot,t),\f_1(\cdot,t)]_{\rm L}(\y)
=\J^{\rm L}\w(\y,t;(\I_{n},\zero)) \f_1(\y,t)-\J^{\rm L}\f_1(\y,t;(\I_{n},\zero))\w(\y,t),\\
[\w,\g_1]_{\rm L}(\y,t)
&:=[\w(\cdot,t),\g_1(\cdot,t)]_{\rm L}(\y)
=\J^{\rm L}\w(\y,t;(\I_{n},\zero)) \g_1(\y,t)-\J^{\rm L}\g_1(\y,t;(\I_{n},\zero))\w(\y,t).
\end{align*}

\begin{proposition}\label{lemma:g2_PWD_t_indep}
Suppose that assumptions (A2)-(A4) hold. Then, given any compact subset $D \subset X$, the function $\g_2(\y)$ is $PWD^{\infty}$ on $D$.
\end{proposition}

\begin{proof}
From assumption (A4), we see that both $\f_1(\y,t)$ and $\f_2(\y,t)$ are $PWC^{\infty}$ in $\y$. Hence, the function $\g_1:X\to\real^n$ defined in \eqref{eq.g1} and the function $\w$ defined in \eqref{eq.w1} are $PWC^{\infty}$. From \cref{lemma.llie}, and from the fact that the sum of $PWC^{\infty}$ and $PWD^{\infty}$ functions is $PWD^{\infty}$, it follows that $\g_2$ is $PWD^{\infty}$ on $D$. 
\end{proof}

Since we are concerned with second-order terms in this subsection, we reintroduce $\Bar{\y}(t)$ to be the solution of the system
\begin{align}\label{eq:Bar_y_with_g2}
    \dot{\Bar{\y}} = \epsilon \g_1(\Bar{\y})+\frac{\epsilon^2}{2}\g_2(\Bar{\y})+ \epsilon^3 \tilde{\g}(\Bar{\y},t,\epsilon),
\end{align}
where $\tilde{\g}(\Bar{\y},t,\epsilon)$ is measurable and essentially bounded. Hence, there exists a Fillipov solution $\Bar{\y}(t;\epsilon)$ of system \eqref{eq:Bar_y_with_g2}.

\begin{theorem}\label{thm:second_order_averaging_BigO_bound}
 Suppose that assumptions (A2)-(A4) hold.
Let $\x(t)=\x(t;\epsilon)$ be the solution of \eqref{eq:average-canonical form_intro} for $\epsilon \in [0,\epsilon_0]$, for some $\epsilon_0>0$, such that $\x(t) \in D$ for some compact set $D \subset X \subseteq \real^n$. Let $\Bar{\y}(t)=\Bar{\y}(t;\epsilon)$ be a solution of system \eqref{eq:Bar_y_with_g2}.
Then the nonsmooth second-order averaging system 
    \begin{align}\label{eq:yExpansion2ndorder_truncated}
     \dot{\zz}_2 = \epsilon \g_1(\zz_2)+\frac{\epsilon^2}{2}\g_2(\zz_2),
 \end{align} 
admits a (Filippov) solution $\zz_2(t)=\zz_2(t;\epsilon)$, which satisfies
\begin{equation}\label{eq:y_bar_g1_g2_BigO}
       \|\Bar{\y}(t;\epsilon)-\zz_2(t;\epsilon)\|=\mathcal{O}(\epsilon),
    \end{equation}
and
    \begin{align}\label{eq:x_bar_y_g1_g2_BigO}
   \|\x(t;\epsilon)-\zz_2(t;\epsilon)\|=\mathcal{O}(\epsilon),
\end{align}
for $t=\mathcal{O}(1/\epsilon)$.

 
\end{theorem}

\begin{proof}
 
From \cref{prop.g1w1}, we see that $\g_1$ is L-smooth, and thus locally Lipschitz continuous.
Also, we see from \cref{lemma:g2_PWD_t_indep} that $\g_2$ is $PWD^{\infty}$, and thus essentially bounded and measurable. 
Hence, there exists a Filippov solution $\zz_2$ of \eqref{eq:yExpansion2ndorder_truncated}.

Next, we note that the difference $\|\g_2(\Bar{\y})-\g_2(\zz_2)\|$ is bounded by some constant $M_{\g_2}$, for all $\Bar{\y}$ and $\zz_2$ in $D$.
Also, we have $M_{\tilde{\g}}=\sup \{\tilde{\g}(\y,t,\epsilon):\y\in D,\epsilon\in(0,\epsilon_0],t\in[0,\frac{\kappa}{\epsilon}]\}$. Then we get the following for $\|\E(t)\|=\|\E(t;\epsilon)\| = \|\Bar{\y}(t;\epsilon)-\zz_1(t;\epsilon)\|$:
\begin{align*}
    \|\E(t)\| &\le \int_0^t \epsilon \|\g_1(\Bar{\y}(\tau)) - \g_1(\zz_1(\tau))\|d\tau +  \epsilon^2 M_{\g_2}t+\epsilon^3 M_{\tilde{\g}}t\\
        &\le \int_0^t \epsilon L_{\g_1}\|\E(\tau)\|d\tau + \epsilon^2 M_{\g_2}t+\epsilon^3 M_{\tilde{\g}}t
        = \epsilon^2 M_{\g_2}t+\epsilon^3 M_{\tilde{\g}}t + \epsilon L_{\g_1}\int_0^t \|\E(\tau)\|d\tau,
\end{align*}
and by using the Specific Gronwall Lemma, we get \eqref{eq:y_bar_g1_g2_BigO}, similar to the proof of \eqref{thm:first_order_averaging}.

Lastly, 
    We see that for $t=\mathcal{O}(1/\epsilon)$, we have:
\begin{align*}
    \|\x(t;\epsilon)-\zz_2(t;\epsilon)\|&\le \|\x(t;\epsilon)-\y(t;\epsilon)\|+\|\y(t;\epsilon) -\Bar{\y}(t;\epsilon)\|+\|\Bar{\y}(t;\epsilon)-\zz_2(t;\epsilon)\|
    =\mathcal{O}(\epsilon).
    \end{align*}
\end{proof}

\begin{remark}\label{remark:computations}
 Using the RHS of \eqref{eq:average-canonical form_intro}, one can compute $\g_1(\zz_2)$ as in \eqref{eq.g1}. To compute $\g_2(\zz_2)$ in \eqref{eq.g2}, one needs to compute $\w(\zz_2,t)$ from \eqref{eq.w1}, which is doable using the already calculated $\g_1(\zz_2)$ and the RHS of \eqref{eq:average-canonical form_intro}; and $\f_2(\zz_2,t)$ from \eqref{eq:xExpansion2ndorder}. For elementary nonsmooth functions, L-derivatives can be computed in closed forms (\cref{subsec:LD}). In \cref{subsec:3.3}, we provide a full worked example for the computation of nonsmooth second-order averaging.  
\end{remark}

\begin{remark}\label{remark:g2_calculated_first_order_bounds_stability}
    Similar to the first case, and as mentioned immediately before \cref{thm:asymptotic_implies_practical},
    \cref{thm:second_order_averaging_BigO_bound} provides the first-order bound needed to prove the stability result in \cref{thm:asymptotic_implies_practical}, but now applied to system \eqref{eq:yExpansion2ndorder_truncated} (the truncated second-order averaging system) instead of system \eqref{eq:yExpansion1storder_truncated}.
\end{remark}
\begin{remark}
    If the RHS of \eqref{eq:average-canonical form_intro} is smooth in $\x$, then $\g_1$ and $\g_2$ in \eqref{eq:yExpansion2ndorder_truncated}  automatically become their smooth version from classical averaging \cite[Equations (24) and (25)]{maggia2020higher}; this is due to the fact that L-derivatives recover classical derivatives if the participating functions are smooth, without needing to check as discussed in \cref{subsec:LD}.
\end{remark}
We move to showing that the trajectories generated by the nonsmooth second-order averaged system \eqref{eq:yExpansion2ndorder_truncated} maintain approximation error bounds analogous to their smooth counterparts \cite{sanders2007averaging,maggia2020higher}. 
 That is, we establish an approximation error bound between $\Bar{\y}$, representing the  ``complete'' averaging system (with higher-order terms), 
and $\zz_2$, the second-order averaging system (truncating any higher-order terms).

 To establish the second-order bounds between the complete averaged solution $\Bar{\y}$ and the second-order truncated solution $\zz_2$, we consider the following assumptions:
 \begin{itemize}
     \item[\textbf{(A5)}]  $\g_1(\y)=\frac{1}{T}\int_0^T \f_1(\y,t)dt=0$ for all $\y$.
     \item[\textbf{(A6)}] The set of discontinuities of $\g_2$, denoted $Z_{\g_2}$, can be written as $Z_{\g_2}=\cup_{i=1}^q Z_i,\;q\in\mathbb{N}$, where $Z_i=\{\y:h_i(\y)=0\}$, where each $h_i$ is a $C^1$ function, and for each $i$: $\nabla h_i(\y)\cdot\g_2(\y)\neq 0,\quad \forall \y\in Z_{i}.$
 \end{itemize}
 Condition (A6) is referred to in the literature (see, for example,  \cite{filippov,bernardo2008piecewise}) as the transversality condition or crossing-hypothesis, and this condition requires  piecewise structure of  the vector field $\g_2$ (which follows from Assumption (A4) and \cref{lemma:g2_PWD_t_indep}).
 We also note that condition (A5) is satisfied in many systems (such as the ESC systems discussed in the next section), and simplifies the expression for $\g_2$, as will be seen from \cref{lemma:g1_zero_simplify_g2} below.

\begin{proposition}\label{lemma:g1_zero_simplify_g2}
Suppose that assumptions (A2)-(A5) hold. Then the function $\g_2$ satisfies
\begin{align}\label{eq:g_2_if_g_1_zero}
        \g_2(\y) = \frac{2}{T} \int_0^T \frac{1}{2}\f_2(\y,t) + \JL_{\y} \f_1(\y,t) \cdot \left(\int_0^t\f_1(\y,s) ds \right) dt.
\end{align}
\end{proposition}

\begin{proof} 
Since $\g_1(\y)=0$ for all $\y$, 
$\w(\y, t) = \int_0^t \f_1(\y, s) ds$.
Recalling the notation  in  \eqref{eq.Lnotation}, 
we note  that  
$
\JL_{\y} \w(\y,t) = \int_0^t  \JL_{\y} \f_1(\y,s) ds,
$
from \cite[Theorem 15]{nesterov2005lexicographic}.
Then we see that we get the following expression for $\g_2(\y)$:
\begin{align*}
\g_2(\y) &= \frac{1}{T}\int_0^T \f_2(\y(t),t) dt + \Gamma (\y),
\end{align*}
where
\begin{align*}
    \Gamma(\y)&= \underbrace{\frac{1}{T}\int_0^T \JL_{\y} \f_1(\y,t)  \left( \int_0^t \f_1(\y, s) ds\right) dt}_{\text{first term}} 
    - \underbrace{\frac{1}{T}\int_0^T \left( \int_0^t  \JL_{\y} \f_1(\y,s) ds \right)  \f_1(\y,t) dt}_{\text{second term}}.
\end{align*}
We simplify the ``second term'' in $\Gamma(\y)$ using integration by parts: let
\begin{align*}
\uu(t)&:=\int_0^t  \JL_{\y} \f_1(\y,s) ds, \quad
\vv(t):=\int_0^t \f_1(\y,s)ds.
\end{align*}
Then $\uu(t)$ and $\vv(t)$ are absolutely continuous; $\vv(t)$ is absolutely continuous because the function $\f_1(\y,t)$ is locally Lipschitz continuous, hence its integral is absolutely continuous.
Let $\widetilde{\J}_{\y}\f_1(\y,t)$ be defined as in \eqref{eq:Jacobian_Like_dropping_zero} (replacing $\f$ and $\x$ in \eqref{eq:Jacobian_Like_dropping_zero} with $\f_1$ and $\y$, respectively), and define $\h(t):= \int_0^t\widetilde{\J}_{\y}\f_1(\y,s)ds$ for some fixed $\y$. Since $\f_1(\y,t)$ is $PWC^{\infty}$ in $\y$, then $\JL_{\y} \f_1(\y,t)=\widetilde{\J}_{\y}\f_1(\y,t)\;\text{a.e.}$ This implies that $\uu(t)=\h(t)$. We see from \cref{lemma:PWSC_implies_PWC.new} that for all $\y,$ $\h(t)$ is absolutely continuous. Hence, $\uu(t)$ is absolutely continuous. Therefore, we may apply integration by parts to yield 
\begin{align*}\frac{1}{T}\int_0^T \left( \int_0^t  \JL_{\y} \f_1(\y,s) ds \right)  \f_1(\y,t) dt
&= \frac{1}{T} \left( \int_0^T  \JL_{\y} \f_1(\y,s) ds \right)  \left( \int_0^T \f_1(\y,\tau) d\tau \right)\\
 &\qquad-  \frac{1}{T}\int_0^T    \JL_{\y} \f_1(\y,t)  \left(\int_0^t\f_1(\y,s) ds \right) dt.
 \end{align*}
Moreover, using $T$-periodicty of $\f_1$ in $t$, we get
$$\frac{1}{T}\int_0^T \left( \int_0^t  \JL_{\y} \f_1(\y,s) ds \right)  \f_1(\y,t) dt=- \frac{1}{T} \int_0^T    \JL_{\y} \f_1(\y,t)  \left(\int_0^t\f_1(\y,s) ds \right) dt.$$
Hence,
\begin{align*}
    \Gamma(\y)&= \frac{2}{T} \int_0^T    \JL_{\y} \f_1(\y,t)  \left(\int_0^t\f_1(\y,s) ds \right) dt,
\end{align*}
and we get
\begin{align*}
    \g_2(\y) = \frac{2}{T} \int_0^T \frac{1}{2}\f_2(\y,t) + \JL_{\y} \f_1(\y,t)  \left(\int_0^t\f_1(\y,s) ds \right) dt.
\end{align*}
\end{proof}
 
The proof of \cref{thm:2nd_order_bound} follows a similar approach to that in \cite[Lemma 8]{llibre2015averaging}.

\begin{theorem}\label{thm:2nd_order_bound}
 Suppose that assumptions (A2)-(A6) hold.  Let $\x(t)=\x(t;\epsilon)$ be the solution of \eqref{eq:average-canonical form_intro} for $\epsilon \in [0,\epsilon_0]$, for some $\epsilon_0>0$, such that $\x(t;\epsilon) \in D$ for some compact set $D \subset X \subseteq \real^n$. Let $\Bar{\y}(t)=\Bar{\y}(t;\epsilon)$ be a solution of system \eqref{eq:Bar_y_with_g2}. Then the solution $\zz_2(t;\epsilon)$ of the system in \eqref{eq:yExpansion2ndorder_truncated}
satisfies
\begin{equation}\label{eq:second_order_bound}
       \|\Bar{\y}(t;\epsilon)-\zz_2(t;\epsilon)\|=\mathcal{O}(\epsilon^2),
    \end{equation}
for $t=\mathcal{O}(1/\epsilon)$.
 \end{theorem}

    \begin{proof}
    Define the following sets for  $i,j \in \{1,\ldots,m\}$:
    \begin{align*}
    A^i_{\Bar{\y},\zz_2}(t)&:=\{\tau\in[0,t]:(\tau,\Bar{\y}(\tau)),(\tau,\zz_2(\tau))\in X_i\},\\
    B^{i,j}_{\Bar{\y},\zz_2}(t)&:=\{\tau\in[0,t]:(\tau,\Bar{\y}(\tau))\in X_i,(\tau,\zz_2(\tau))\in X_j, i\neq j\},\\
    C_{\Bar{\y}}(t)&:=\{\tau\in[0,t]: (\tau,\Bar{\y}(\tau))\in Z_{\g_2}\},\quad
    D_{\zz_2}(t):=\{\tau\in[0,t]: (\tau,\zz_2(\tau))\in Z_{\g_2}\},
    \end{align*}
     where $X_i,\; i \in \{1,\ldots,M\}$ are the open sets from the $PWC^{\infty}$ decomposition of $\g_2$ (i.e., $\g_2$ is $C^{\infty}$ on each $X_i$).
     From the crossing hypothesis (A6), we see that the measures of both $C_{\Bar{\y}}(t)$ and $D_{\zz_2}(t)$, denoted by  $\mu(C_{\Bar{\y}}(t))$ and $\mu(D_{\zz_2}(t))$, respectively, are zero. Let $\|\E(t)\|=\|\E(t;\epsilon)\|=\|\Bar{\y}(t;\epsilon)-\zz_2(t;\epsilon)\|$.
    First, we prove that
    for some constants $\alpha_1,\alpha_2,\alpha_3$, we have 
    \begin{equation}\label{eq:claim1}
    \begin{split}
        \|\E(t)\|&\le\int_0^t\frac{\epsilon^2}{2}\alpha_1\|\E(\tau)\|d\tau + \sum_{i=1}^M\int_0^t\frac{\epsilon^2}{2} \alpha_2\|\E(\tau)\|d\tau +\sum_{i=1}^M\sum_{j=1}^M\frac{\epsilon^2}{2}  \alpha_3\mu(B^{i,j}_{\Bar{\y},\zz_2}(t)).
    \end{split}
    \end{equation}

    Observe that
    \begin{align*}
         &\|\E(t)\|
         =\Big|\Big|\int_0^t \epsilon \Bar{\f}_1(\Bar{\y}(\tau)) - \Bar{\f}_1(\zz_2(\tau)) + \frac{\epsilon^2}{2}\g_2(\Bar{\y}(\tau))-\g_2(\zz_2(\tau)) + \epsilon^3  \hat{\g}(\Bar{\y}(\tau),\tau,\epsilon)d\tau\Big|\Big|\\
         &\quad \le \underbrace{\int_0^t \epsilon \|\Bar{\f}_1(\Bar{\y}(\tau)) - \Bar{\f}_1(\zz_2(\tau))\|d\tau}_{I_1}+ \underbrace{\int_0^t\frac{\epsilon^2}{2}\|\g_2(\Bar{\y}(\tau))-\g_2(\zz_2(\tau))\|d\tau}_{I_2}+ \underbrace{\epsilon^3 \int_0^t \hat{\g}(\Bar{\y}(\tau),\tau,\epsilon)d\tau}_{I_3},
    \end{align*}
    using Minkowski's integral inequality. Now we proceed by calculating $I_1,I_2$, and $I_3$. For $I_1$, we note that since $\f_{1}(\Bar{\y},t)$ is Lipchitz continuous with respect to $\Bar{\y}$ on $D$, then $\Bar{\f}_1(\Bar{\y})$ is Lipchitz continuous with respect to $\Bar{\y}$ on  $D$ with some Lipchitz constant $L_{\Bar{\f}_1}$, and we have 
    \begin{align*}
        I_1(t) &= \int_0^t \epsilon \|\Bar{\f}_1(\Bar{\y}(\tau)) - \Bar{\f}_1(\zz_2(\tau))\|d\tau \le   \epsilon\int_0^t L_{\Bar{\f}_1}\|\E(\tau)\|d\tau. 
    \end{align*}
    For $I_3$, and recalling that $\hat{\g}$ is bounded on a compact set $D$, we see that
    \begin{align*}
        I_3(t) &= \epsilon^3 \int_0^t \hat{\g}(\Bar{\y}(\tau),\tau,\epsilon)d\tau
        \le \epsilon^3 M_{\hat{\g}} t,
    \end{align*}
    where $M_{\hat{\g}}$ is a constant. For $I_2$, from the properties of $\f$, we note that:
    \begin{enumerate}
        \item[(P1)] The function $\f_1$ is bounded and integrable on $D$ because $\f_1$ is continuous. 
        \item[(P2)]  The function $\JL \f_1=\JL \f_1(\cdot,t;\M,\mathbf{0})$ is bounded above on $D$ because  $\JL \f_1$ is $PWD^{\infty}$ (from \cref{prop:LDerivative_PWC_PWD}).
         
        \item [(P3)]   Let $J_{\Bar{\y}}=\JL_{\Bar{\y}} \f_1(\Bar{\y},t)$ and  $I_{\Bar{\y}}=\int_0^t\f_1(\Bar{\y},s) ds$. We see that $I_{\Bar{\y}}$ is bounded for all $\Bar{\y}\in D$ and all $\tau\in[0,t]$ (since $\f_1$ is bounded), i.e.,  $\|I_{\Bar{\y}}\|\le M_{I},\:\forall \Bar{\y}\in D, \tau\in[0,t]$. 
        Similarly, $\|J_{\Bar{\y}}\|\le M_{J},\:\forall \Bar{\y}\in D, \tau\in[0,t]$. 
        \item[(P4)]  We see that for all $\Bar{\y},\zz_2\in D,t\in[0,t]$, $\|I_{\Bar{\y}} J_{\Bar{\y}}-I_{\zz_2} J_{\zz_2} \|=\|I_{\Bar{\y}} J_{\Bar{\y}}-I_{\Bar{\y}}J_{\zz_2}-I_{\zz_2} J_{\zz_2}+I_{\Bar{\y}} J_{\zz_2}\|\le \|I_{\Bar{\y}}\| \|J_{\Bar{\y}}-J_{\zz_2}\|+\|J_{\zz_2}\|\|I_{\Bar{\y}}-I_{\zz_2} \| \le M_{I} \|J_{\Bar{\y}}-J_{\zz_2}\| + M_{J} \|I_{\Bar{\y}}-I_{\zz_2} \|$. 
        \item[(P5)] For any two vectors $\Bar{\y}$ and $\zz_2$ in $D$, the difference $\|\JL_{\Bar{\y}} \f_1(\Bar{\y},t)-\JL _{\zz_2}\f_1(\zz_2,t)\|, \:\forall \tau\in[0,t]$ is bounded above by some constant $M_{\JL \f_1}$.
        \item[(P6)]  We note that the function $I_{\Bar{\y}}$ is Lipschitz continuous with respect to $\Bar{\y}$ on $D$, with Lipschitz constant $L_{I \f_1}$. 
        {This follows from the Lipschitz continuity of $\f_1$ with respect to $\Bar{\y}$.}
        \item[(P7)] Also, the function $J_{\Bar{\y}}$ is Lipschitz continuous with respect to $\Bar{\y}$ on $X_i\cap D$ , for all $i=\{1,\dots,M\}$, with Lipchitz constants $L_{\JL \f_{1,i}}$.
        \item [(P8)] The function $\f_2(\Bar{\y},t)$ is Lipschitz continuous with respect to $\Bar{\y}$ on $D$, with  Lipschitz constant $L_{\f_2}$ (because $\f$ is L-smooth with respect to $\Bar{\y}$).
    \end{enumerate}
  \allowdisplaybreaks
    Hence, we have
    \begin{align*}
        I_2(t) &= \int_0^t\frac{\epsilon^2}{2}\|\g_2(\Bar{\y}(\tau))-\g_2(\zz_2(\tau))\|d\tau \\
        &= \int_0^t\frac{\epsilon^2}{2}\Big|\Big|\frac{2}{T} \int_0^T \frac{1}{2}\f_2(\Bar{\y}(\tau),r) + \JL_{\Bar{\y}} \f_1(\Bar{\y}(\tau),r)\cdot \int_0^r\f_1(\Bar{\y}(\tau),s) ds  dr
        \\
        &\qquad\qquad\quad-\frac{2}{T} \int_0^T \frac{1}{2}\f_2(\zz_2(\tau),r)+ \JL_{\zz_2} \f_1(\zz_2(\tau),r) \cdot \left(\int_0^r\f_1(\zz_2(\tau),s) ds \right) dr\Big|\Big|d\tau
        \\
        &\le \int_0^t\frac{\epsilon^2}{2}\frac{1}{T} \int_0^T \| \f_2(\Bar{\y}(\tau),r)- \f_2(\zz_2(\tau),r)\|drd\tau
        \\
        &\qquad+\int_0^t\frac{\epsilon^2}{2}\frac{1}{T} \int_0^T \Big|\Big|\JL_{\Bar{\y}} \f_1(\Bar{\y}(\tau),r) \cdot \left(\int_0^r\f_1(\Bar{\y}(\tau),s) ds \right) - \JL_{\zz_2} \f_1(\zz_2(\tau),r) \cdot \left(\int_0^r\f_1(\zz_2(\tau),s) ds \right)\Big|\Big| drd\tau\\
        &\le\int_0^t\frac{\epsilon^2}{2}\frac{1}{T} \int_0^T L_{\f_2}\|\E(\tau)\|drd\tau +\int_0^t\frac{\epsilon^2}{2}\frac{1}{T} \int_0^T \Big[ M_I\| \JL_{\Bar{\y}} \f_1(\Bar{\y}(\tau),r)
         - \JL_{\zz_2} \f_1(\zz_2(\tau),r)\|\qquad\text{(from (P4))}\\
         &\qquad\qquad\qquad\qquad\quad+  M_{J}\Big|\Big|\int_0^r\f_1(\Bar{\y}(\tau),s) ds -\int_0^r\f_1(\zz_2(\tau),s) ds \Big|\Big| \Big] drd\tau\\
        &\le\int_0^t\frac{\epsilon^2}{2}\frac{1}{T} \int_0^T L_{\f_2}\|\E(\tau)\|drd\tau+\int_0^t\frac{\epsilon^2}{2}\frac{1}{T} \int_0^T \Big[ M_I\| \JL_{\Bar{\y}} \f_1(\Bar{\y}(\tau),r)
         - \JL_{\zz_2} \f_1(\zz_2(\tau),r)\|\\
         &\qquad\qquad+  \int_0^t\frac{\epsilon^2}{2}\frac{1}{T} \int_0^T M_{J}L_{I\f_1}\|\E(\tau)\| drd\tau\quad\text{(from (P6))}\\
         &=\int_0^t\frac{\epsilon^2}{2}(L_{\f_2}+M_{J}L_{I\f_1})\|\E(\tau)\|d\tau+\frac{\epsilon^2}{2}\frac{M_{I}}{T}\Big[\sum_{i=1}^M\int_{A^i_{\Bar{\y},\zz_2}(t)}\int_0^T  L_{\JL \f_1,i}\|\E(\tau)\| dr d\tau\quad\text{(from (P7))}\\
         &\qquad\qquad\quad+\sum_{i=1}^M\sum_{j=1}^M\int_{B^{i,j}_{\Bar{\y},\zz_2}(t)} \int_0^T \|\JL_{\Bar{\y}} \f_1(\Bar{\y}(\tau),r)- \JL_{\zz_2} \f_1(\zz_2(\tau),r)\|   dr d\tau\\
         &\qquad\qquad\quad+\int_{C_{\Bar{\y}}(t)} \int_0^T2 M_{J} dr d\tau
         +\int_{D_{\Bar{\y}}(t)} \int_0^T 2 M_{J} dr d\tau\Big]\quad\text{(from (P3))}\\
          &=\frac{\epsilon^2}{2}\Bigg[\int_0^t(L_{\f_2}+M_{J}L_{I\f_1})\|\E(\tau)\|d\tau +M_{I}\Big[\sum_{i=1}^M\int_{A^i_{\Bar{\y},\zz_2}(t)}\frac{\epsilon^2}{2}L_{\JL \f_1,i} \|\E(\tau)\|d\tau \\
         &\;+\sum_{i=1}^M\sum_{j=1}^M \int_{B^{i,j}_{\Bar{\y},\zz_2}(t)}\frac{1}{T}\int_0^T \|\JL_{\Bar{\y}} \f_1(\Bar{\y}(\tau),r)- \JL_{\zz_2} \f_1(\zz_2(\tau),r)\|   drd\tau+2 M_{J} \underbrace{\mu(C_{\Bar{\y}}(t))}_{=0}+ 2M_{J} \underbrace{\mu(D_{\Bar{\y}}(t))}_{=0}\Big]\Bigg]\\
         &\le\frac{\epsilon^2}{2}\Big[\int_0^t(L_{\f_2}+M_{J}L_{I\f_1})\|\E(\tau)\|d\tau
         + \sum_{i=1}^M\int_0^t M_{I}L_{\JL \f_1,i}\|\E(\tau)\|d\tau +\sum_{i=1}^M\sum_{j=1}^M  M_{\JL \f_1}M_{I}\mu(B^{i,j}_{\Bar{\y},\zz_2}(t))\Big].
        \end{align*}
        \normalsize
        Now, we want to show that $\mu(B^{i,j}_{\Bar{\y},\zz_2}(t))=\mathcal{O}(\epsilon)$.
        Suppose there exists a connected subset $[t_1,t_2]\subseteq B_{\Bar{\y},\zz_2}(t)$, i.e., for all $t\in [t_1,t_2]$, we have $\Bar{\y}(t)\in X_i$, $\zz_2(t)\in X_j$, $i\neq j$, and for some $\epsilon>0$, we have $\zz_2(t_1-\epsilon)\in X_i$ and  $\Bar{\y}(t_2+\epsilon)\in X_j$. 
        
        {To show $\mu(B^{i,j}_{\Bar{\y},\zz_2}(t))=\mathcal{O}(\epsilon)$, we aim to show $\|t_1-t_2\|<C_1\epsilon$ for some constant $C_1>0$.}
        Let $t^*\in[t_1,t_2]$ such that $\Bar{\y}^*:=\zz_2(t^*)\in Z_k$, for some $k\in\mathbb{N}$ (note that $Z_k\subseteq Z_{\g_2}$ from Assumption (A6)). The crossing hypothesis implies that
        \begin{align}\label{eq:transv_cond}
           \nabla h_k(\Bar{\y}^*)\cdot \g_2(\Bar{\y}^*) \neq 0.
        \end{align}
        Consider $t_a<t_1$ such that  $\zz_2(t_a)\in X_i$ and $\zz_2(t_2)\in X_j$, with $i\neq j$.  It can be seen from the continuity of $h_k$, the assumption $h_k(\Bar{\y}^*)=0$ if and only if $\Bar{\y}^*\in Z_k$, and equation \eqref{eq:transv_cond}, that the value of $h_k$ on $X_i$ has a different sign from the value of $h_k$ on $X_j$. Otherwise, there exists a $\tilde{\Bar{\y}}\notin Z_k$ such that $h_k(\tilde{\Bar{\y}})=0$, contrary to our assumption.  In other words, $h_k$ is strictly monotone on $[t_a,t_2]$, and $\frac{dh_k}{dt}(\zz_2(t))\neq 0$ for all $t\in [t_a,t_2]$. Define $H(t):=h_k(\zz_2(t))$.  {Then  from the Mean Value Theorem there exists $\tilde{t}\in [t_a,t_2]$ such that }
        \begin{align*}
             H(t_2)- H(t_a) = \lambda \|t_2-t_a\| 
        \end{align*}

        where $\lambda=\frac{dH}{dt}(\tilde{t})\neq 0$. That is 
        $$\|t_2-t_1\|\le\|t_2-t_a\|\le\frac{\|h_k(\zz_2(t_2))- h_k(\zz_2(t_a))\|}{\lambda}.$$ 
        Hence, we established the inequality relation between $\|t_2-t_1\|$ and the difference $h_k(\zz_{21})-h_k(\zz_{22})$, for any $\zz_{21}\in X_i$ and $\zz_{22}\in X_j$. We also note that the function $h_k$ is Lipschitz continuous on $D$, hence we have
        \begin{align*}
            \|h_k(\Bar{\y}(t))-h_k(\zz_2(t))\|\le L_{h_k} \|\Bar{\y}(t)-\zz_2(t)\|\le L_{h_k} C\epsilon.
        \end{align*}
         Since $\Bar{\y}(t)\in X_i$, and $\zz_2(t)\in X_j$, we see that 
         \begin{align*}
            \|t_2-t_1\|\le \frac{\|h_k(\Bar{\y}(t))-h_k(\zz_2(t))\|}{\lambda}\le\frac{L_{h_k}C\epsilon}{\lambda}=C_1\epsilon,
         \end{align*}
         where $C_1=\frac{L_{h_k}C}{\lambda}$
     Since $[t_1,t_2]$ is chosen arbitrarily from $B^{i,j}_{\Bar{\y},\zz_2}(t)$, then $\mu(B^{i,j}_{\Bar{\y},\zz_2}(t))=\mathcal{O}(\epsilon)$, as desired.
     
    Next, let us 
    consider a constant $M_B$ such that $ \sum_{i=1}^M\sum_{j=1}^MM_{\JL \f_1}M_{I}<M_B$, then we have from \eqref{eq:claim1} that:
    \begin{align*}
        \|\E(t)\|&\le\int_0^t \epsilon L_{\Bar{\f}_1}\|\E(\tau)\|d\tau 
        +\int_0^t\frac{\epsilon^2}{2}(L_{\f_2}+M_{J}L_{I\f_1})\|\E(\tau)\|d\tau +\sum_{i=1}^M\int_0^t\frac{\epsilon^2}{2}M_{I}L_{\JL \f_1,i}\|\E(\tau)\|d\tau\\
        &\qquad+ \epsilon^3 M_{\hat{\g}} t+\sum_{i=1}^M\sum_{j=1}^M\frac{\epsilon^2}{2}  M_{\JL \f_1}M_{\f_1}\mu(B^{i,j}_{\Bar{\y},\zz_2}(t))\\
         &\le \int_0^t \epsilon L_{\Bar{\f}_1}\|\E(\tau)\|d\tau 
        +\int_0^t\frac{\epsilon^2}{2}(L_{\f_2}+M_{J}L_{I\f_1})\|\E(\tau)\|d\tau
        + \sum_{i=1}^M\int_0^t\frac{\epsilon^2}{2}M_{I}L_{\JL \f_1,i}\|\E(\tau)\|d\tau \\
        &\qquad\qquad+ \epsilon^3 M_{\hat{\g}} t +  \epsilon^3 M_B\\
        &\le \epsilon^3 (M_{\hat{\g}}+M_B) t + \epsilon(1+\frac{\epsilon}{2}+\frac{\epsilon}{2}M_{L})L_m \int_0^t\|\E(\tau)\|d\tau,  
    \end{align*}
where $L_m=\max( L_{\Bar{\f}_1},L_{\f_2},\sum_{i=1}^ML_{\JL \f_1,i})$.

        Finally, we get from Gronwall's Lemma that:
    \begin{align*}
        \|\E(t)\|&\le   \frac{\epsilon^3(M_{\hat{\g}}+M_B)}{\epsilon(1 +\frac{\epsilon}{2}+\frac{\epsilon}{2}M_{L}) L_{m}} e^{\epsilon(1+\frac{\epsilon}{2}+\frac{\epsilon}{2}M_{L}) L_m t}- \frac{\epsilon^3 (M_{\hat{\g}}+M_B)}{\epsilon(1+\frac{\epsilon}{2}+\frac{\epsilon}{2}M_{L}) L_m} .
    \end{align*}
    As $\epsilon\rightarrow 0$, then $1+\frac{\epsilon}{2}+\frac{\epsilon}{2}M_{L}\rightarrow 1$. Hence,
    \begin{align*}
         \|\E(t)\|\le \epsilon^2 \frac{(M_{\hat{\g}}+M_B)}{L_m}(e^{\epsilon L_m t}-1).
    \end{align*}
    With $t=\mathcal{O}(1/\epsilon)$, we get $\|\E(t;\epsilon)\|=\|\Bar{\y}(t;\epsilon)-\zz_2(t;\epsilon)\|=\mathcal{O}(\epsilon^2)$.
\end{proof}

\begin{remark}\label{rmk:closeness_of_trajectories}
    As is the case in smooth near-identity transformation (classical) averaging \cite{sanders2007averaging,maggia2020higher}, the approximation error (closeness of trajectories) between the nonsmooth NLTP system \eqref{eq:average-canonical form_intro} and its nonsmooth averaged system \eqref{eq:yExpansion2ndorder_truncated} holds (i.e., \cref{thm:2nd_order_bound} and \cref{eq:x_bar_y_g1_g2_BigO} hold) regardless of stability; this will be demonstrated by example/simulations in \cref{subsec:3.3}.    
\end{remark}

\subsection{Summary of Averaging Systems}\label{subsec:3.25}
In this part, we summarize the averaging theory above --- see \cref{table.averaging}.
 \begin{table}[h!]
   \hspace{-2cm}
    \renewcommand{\arraystretch}{1.5} 
    \begin{tabular}{|m{5cm}|m{3cm}|m{3cm}|m{3cm}|m{3cm}|}
        \hline
        \textbf{System Equations} & \textbf{Equation Label}  & \textbf{Type of System} & \textbf{Type of Solution} & \textbf{Relevant Theorems} \\
        \hline
        \begin{equation*}
             \dot{\x} = \epsilon \f(\x,t,\epsilon)
        \end{equation*} & \eqref{eq:average-canonical form_intro}  & NLPT & Classical & --  \\
        \hline
        \begin{align*}
             \dot{\x} = \epsilon \f_1(\x,t)+\frac{\epsilon^2}{2!}\f_2(\x,t) +\frac{\epsilon^3}{3!}\hat{\f}(\x,t,\epsilon)
        \end{align*} &  \eqref{eq:xExpansion2ndorder}  & Expanded NLTP & Classical  & -- \\
        \hline
        \begin{equation*}
     \dot{\y} = \epsilon \h_1(\y,t)+\epsilon^2 \hat{\h}(\y,t,\epsilon)
 \end{equation*} & \eqref{eq:yExpansion2ndorder} & Expanded NLTP under Near-Identity Transformation & Filippov  & \cref{lemma:g_1_g_2_hatg_Loc_Lip}  \\
        \hline  
        \begin{equation*}
     \dot{\Bar{\y}} = \epsilon \g_1(\Bar{\y}) + \epsilon^2 \bgHat(\Bar{\y}, t, \epsilon)
 \end{equation*} & \eqref{eq:averaged_sys}  & Full Averaging of NLTP & Filippov  &  \cref{thm:y_bar_y_bounds} \\
        \hline  
        \begin{equation*}
    \dot{\zz}_1 = \epsilon \Bar{\f}_1(\zz_1)
 \end{equation*}
 
 & \eqref{eq:yExpansion1storder_truncated}  & First-Order Averaging of NLTP  & Classical &  \cref{thm:first_order_averaging}  \\
        \hline
        \begin{equation*}
   \dot{\Bar{\y}} = \epsilon \g_1(\Bar{\y})+\frac{\epsilon^2}{2}\g_2(\Bar{\y})+ \epsilon^3 \tilde{\g}(\Bar{\y},t,\epsilon)
 \end{equation*}
 & \eqref{eq:Bar_y_with_g2}  & Full Averaging of NLTP  & Filippov&   \cref{thm:2nd_order_bound}  \\
        \hline
        \begin{equation*}
  \dot{\zz}_2  = \epsilon \g_1(\zz_2)+\frac{\epsilon^2}{2}\g_2(\zz_2)
 \end{equation*}
 
 & \eqref{eq:yExpansion2ndorder_truncated}  & Second-Order Averaging of NLTP & Filippov&  \cref{thm:2nd_order_bound}  \\
        \hline
    \end{tabular}
    \caption{Summary of nonsmooth averaging theory.}\label{table.averaging}
\end{table}

 \subsection{Nonsmooth averaging example}\label{subsec:3.3}
In this part, we provide an example of an NLTP system in the averaging canonical form \eqref{eq:average-canonical form_intro} and derive its second-order averaging using LD-derivatives, as described above. In fact, we choose a system that is not stable to verify  \cref{eq:x_bar_y_g1_g2_BigO} and illustrate the observations made in  \cref{rmk:closeness_of_trajectories}. 

Consider the following  system:
\begin{equation}\label{eq:divergent_multiagent}
\begin{split}
\frac{dx_1}{dt}=cJ(\x)\sqrt{\omega}\sin(\omega t)+\alpha\sqrt{\omega}\cos(\omega t),\quad
\frac{dx_2}{dt}=-cJ(\x)\sqrt{\omega}\cos(\omega t)+\alpha\sqrt{\omega}\sin(\omega t),
\end{split}
\end{equation}
     with state variables $\x=(x_1,x_2)$,  $J(\x)=|x_1|+|x_2|$, and and constants $c$ and $\alpha$. First, in order to apply \cref{thm:second_order_averaging_BigO_bound} to \eqref{eq:divergent_multiagent}, the system \eqref{eq:divergent_multiagent} is required to be in the averaging-canonical form as in \eqref{eq:average-canonical form_intro}. Accordingly, let $\tau = \omega t$ and $\epsilon = \frac{1}{\sqrt{\omega}}$,  from which it follows that 
     $$\frac{d x_1}{d\tau} = \frac{1}{\omega}\frac{dx_1}{dt}, \quad \frac{d x_2}{d\tau} = \frac{1}{\omega} \frac{dx_2}{dt}.$$ 
     Hence, we get the NLTP system:      \begin{equation}\label{eq:divergent_multiagent.NLTP}
    \begin{aligned}
         \frac{d x_1}{d\tau}&=\epsilon[cJ(\x)\sin(\tau)+\alpha\cos(\tau)],\quad
        \frac{d x_2}{d\tau}=\epsilon[-cJ(\x)\cos(\tau)+\alpha\sin(\tau)],
     \end{aligned}
     \end{equation}  
which is in the averaging-canonical form $\dot{\x}=\epsilon \f(\x,t,\epsilon)$ in \eqref{eq:average-canonical form_intro}. Now, we follow \cref{remark:computations} and build the averaging systems:  first,
\begin{align*}
    \g_1(\y) &=\bar{\f}_1(\y)=  \frac{1}{2\pi}\int_0^{2\pi} \begin{bmatrix}cJ(\y)\sin(s)+\alpha\cos(s)\\-cJ(\y)\cos(s)+\alpha\sin(s)\end{bmatrix}ds = \zero.
\end{align*}
Next, to construct the second-order term $\g_2(\y)$,  
observe that 
\begin{equation}\label{eq:w1}
\begin{split}
    \w(\y,\tau)  
    = \int_{0}^{\tau} \left(\f_1(\y,s) - \bar{\f}_1(\y\right) ds 
    &=\int_0^{\tau} \begin{bmatrix}cJ(\y)\sin+\alpha\cos(s)\\-cJ(\y)\cos(s)+\alpha\sin(s)\end{bmatrix}ds=\begin{bmatrix}-cJ(\y)\cos(\tau)+cJ(\y)+\alpha\sin(\tau)\\-cJ(\y)\sin(\tau)-\alpha\cos(\tau)+\alpha\end{bmatrix}.
\end{split}
\end{equation}
Moreover, 
\begin{align*}
    [\w,\f_1]_{\rm L}(\y,\tau)&= \JL\f_1(\y,\tau)
    \begin{bmatrix}-cJ(\y)\cos(\tau)+cJ(\y)+\alpha\sin(\tau)\\-cJ(\y)\sin(\tau)-\alpha\cos(\tau)+\alpha\end{bmatrix}- \JL\w(\y,\tau)  \begin{bmatrix}cJ(\y)\sin(\tau)+\alpha\cos(\tau)\\-cJ(\y)\cos(\tau)+\alpha\sin(\tau)\end{bmatrix}.
\end{align*} 
Then, using the sharp calculus rules of LD-derivatives in \eqref{eq:sharp_rules}, we get 
\begin{equation} \label{eq:J_L f_1}
\begin{aligned}
    \JL\f_1(\y,\tau) 
    &=\begin{bmatrix}
        c\sin(\tau)[J]'(\y)\\
        -c\cos(\tau)[J]'(\y)
    \end{bmatrix}
     = \begin{bmatrix}
        c\sin(\tau)\mathrm{fsign}(x_1,1)&c\sin(\tau)\mathrm{fsign}(x_2,1)\\
        -c\cos({\tau})\mathrm{fsign}(x_1,1)&-c\cos({\tau})\mathrm{fsign}(x_2,1)
    \end{bmatrix}.
\end{aligned}
\end{equation}
where, via \eqref{eq.LDderivative.abs}, 
\begin{align*}
 [J]'(\y)&=J'(\y;\II_n)=\nabla_L J(\y)
    = 
    \mathrm{fsign}(x_1,[1\quad 0])[1\quad 0]
    +\mathrm{fsign}(x_2,[0\quad 1])[0\quad 1]\\
    &=[\mathrm{fsign}(x_1,[1\quad 0]) \quad \mathrm{fsign}(x_2,[0\quad 1])]
    =[\mathrm{fsign}(x_1,1) \quad \mathrm{fsign}(x_2,1)].
\end{align*}
Similarly,  
\begin{align}\label{eq:J_L w1}
    \JL\w(\y,\tau) &= [\w]'(\y,\tau)=[\w]'(\y,\tau;(\II_n,\zero))
    = \begin{bmatrix}
        \beta_1&\beta_2\\
        \beta_3&\beta_4
    \end{bmatrix},
\end{align}
where 
\begin{align*}
\beta_1&=-c\cos(\tau)\mathrm{fsign}(x_1,1)+c\mathrm{fsign}(x_1,1),\quad
\beta_2=-c\cos(\tau)\mathrm{fsign}(x_2,1)+c\mathrm{fsign}(x_2,1),\\ 
\beta_3&= -c\sin({\tau})\mathrm{fsign}(x_1,1),\quad
\beta_4=-c\sin({\tau})\mathrm{fsign}(x_2,1).
\end{align*}
Since $\g_1(\y)=\zero$, then it follows that $[\w,\g_1]_{\rm L}(\y,t)=\zero$. 
Also, $\f_2(\x,t)=  2\frac{\partial \f}{\partial \epsilon}(\x,t,0) = \zero$. Hence, from  \eqref{eq:w1}, \eqref{eq:J_L f_1} and \eqref{eq:J_L w1} we get 
\begin{align*}
    \g_2(\y)
    &=\frac{1}{2\pi}\int_{0}^{2\pi} (\f_2(\y,t) + [\w,\f_1]_{\rm L}(\y,t)+[\w,\g_1]_{\rm L}(\y,t))dt\\
     &= \frac{1}{2\pi}\int_0^{2\pi} \begin{bmatrix}
        c\alpha\mathrm{fsign}(x_1,1)-c^2\mathrm{fsign}(x_2;1)J(\y)\\
        c^2\mathrm{fsign}(x_1;1)J(\y)+\alpha c\mathrm{fsign}(x_2,1)
    \end{bmatrix}dt\\
    &\quad - \frac{1}{2\pi}\int_0^{2\pi}\left(\begin{bmatrix}
        c\alpha\mathrm{fsign}(x_1,1)\cos(t)\\0\end{bmatrix} +\begin{bmatrix}c^2\mathrm{fsign}(x_2,1)J(\y)\cos(t)\\
        0
    \end{bmatrix}\right) dt\\
     &= \begin{bmatrix}
        c\alpha\mathrm{fsign}(x_1,1)-c^2\mathrm{fsign}(x_2,1)(|x_1|+|x_2|)\\
        c^2\mathrm{fsign}(x_1,1)(|x_1|+|x_2|)+ c\alpha\mathrm{fsign}(x_2,1)
    \end{bmatrix}.
\end{align*}
Then, we get that the second-order averaging system is given by 
\begin{align*}
    \frac{d\zz}{d\tau} 
    &= \epsilon \g_1 (\zz)+\frac{\epsilon^2}{2!} \g_2 (\zz)
    =  \frac{\epsilon^2}{2}\begin{bmatrix}
        c\alpha\mathrm{fsign}(z_1,1)-c^2\mathrm{fsign}(z_2,1)J(\zz)\\
        c^2\mathrm{fsign}(z_1,1)J(\zz)+ c\alpha\mathrm{fsign}(z_2,1)
    \end{bmatrix}\\
    &= \frac{1}{2\omega}\begin{bmatrix}
            c\alpha &-c^2J(\zz)\\
            c^2J(\zz) & c\alpha
        \end{bmatrix}\begin{bmatrix}
            \mathrm{fsign}(z_1,1)\\
            \mathrm{fsign}(z_2,1)
        \end{bmatrix}
    =\frac{1}{2\omega}\A(\zz)\begin{bmatrix}
            \mathrm{fsign}(z_1,1)\\
            \mathrm{fsign}(z_2,1)
        \end{bmatrix}.
\end{align*}
where
$$\A(\zz)=\begin{bmatrix}
            c\alpha &-c^2J(\zz)\\
            c^2J(\zz) & c\alpha
        \end{bmatrix}.$$
Finally, the second-order averaging system  in the timescale of the original variable $t$ is 
\begin{align}\label{eq:second_order_example}
\begin{split}
    \frac{d\zz}{dt} = \omega \frac{d\zz}{d\tau}&= \frac{1}{2}\A(\zz)\begin{bmatrix}
            \mathrm{fsign}(z_1,1)\\
            \mathrm{fsign}(z_2,1)
        \end{bmatrix}.
\end{split}
\end{align}

\cref{fig:diverge_x1_x2} shows a simulation (codes are available in \cite{github_averaging_ESC}) for the original nonsmooth time-varying system \eqref{eq:divergent_multiagent} and its second-order averaged system \eqref{eq:second_order_example}, which is not stable. However, as predicted by \cref{eq:x_bar_y_g1_g2_BigO}, closeness of the trajectories is maintained between $\x$ (solid blue) and $\zz$ (dashed red) even though both are diverging; that is, $\|\x(t;\epsilon)-\zz(t;\epsilon)\|=\mathcal{O}(\epsilon)$. We also notice how the L-smooth averaged system captures multiple nonsmooth turns throughout the simulation. In our simulation, we have $\x_0=\zz_0=\begin{bsmallmatrix}
   2\\-2 
\end{bsmallmatrix}$, $\omega=10$, $c=0.3$ and $\alpha=1$.

\begin{figure}[htbp]
    \centering
    \hspace*{-0.5cm} 
\includegraphics[width=0.85\textwidth]{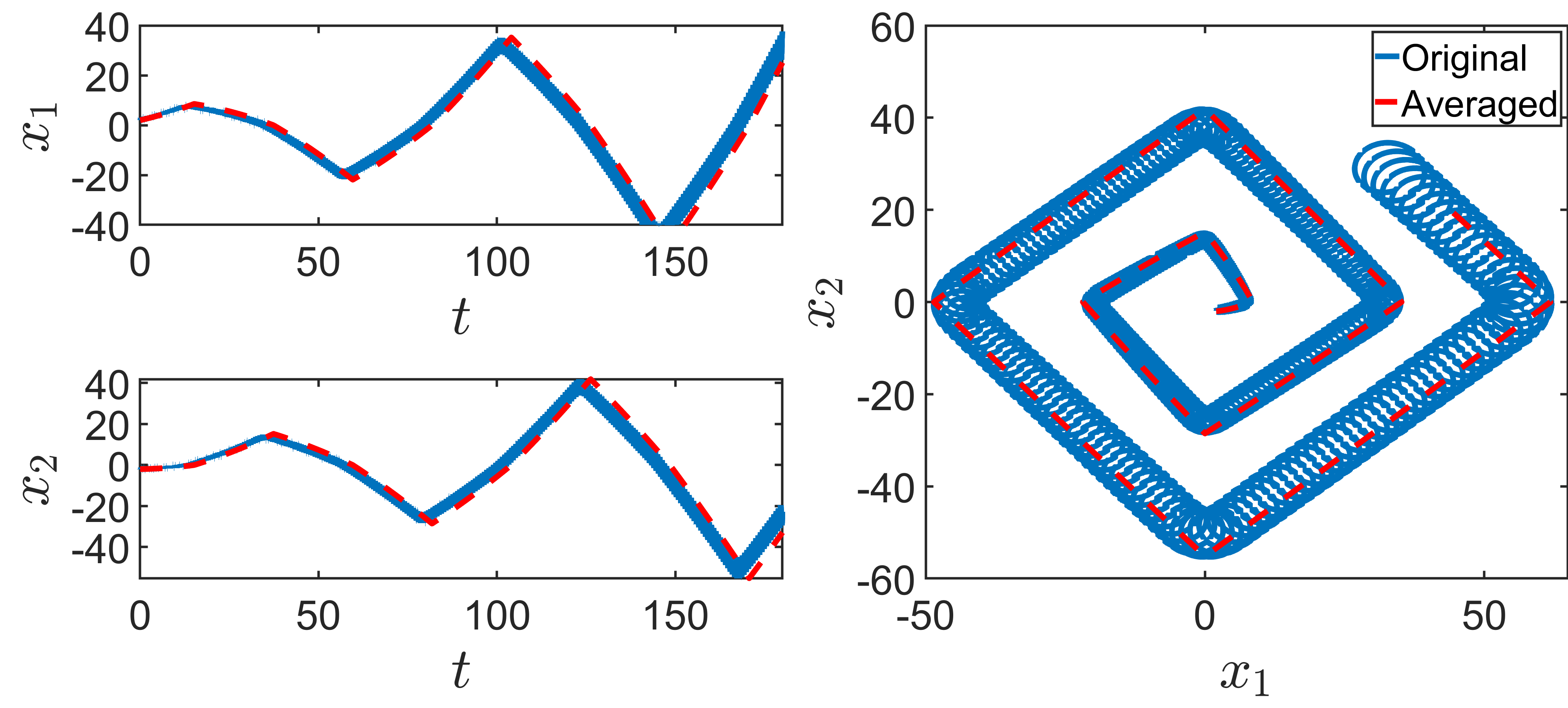}
\caption{Simulation results of  example with initial conditions $\x_0=\zz_0=\begin{bsmallmatrix}
   2\\-2 
\end{bsmallmatrix}$, $\omega=10$, $c=0.3$ and $\alpha=1$. The two left figures show the states $x_1$ and $x_2$ (solid blue) with their respective second-order averages $\zz_{2,1}$ and $\zz_{2,2}$ (dashed red), while the right figure shows the trajectories $x_1$ vs $x_2$ (solid blue) with $\zz_{2,1}$ and $\zz_{2,2}$ (dashed red). 
} 
\label{fig:diverge_x1_x2}
\end{figure}

\section{Nonsmooth control-affine extremum seeking systems}\label{sec:NonsmoothESC}
 
In this part, we use the nonsmooth averaging theory to achieve results in nonsmooth extremum seeking control (ESC) systems.

\subsection{Generalized-gradient-based extremum seeking}\label{subsection:4.1}

Let us consider the control-affine ESC system below:
\begin{align}\label{eq:control_affine}
\frac{d\x}{dt}=\bm{b}_1(\x)u_1(t)+\bm{b}_2(\x)u_2(t),
\end{align}
where $\bm{b}_1(\x)$ and $\bm{b}_2(\x)$ are vector fields that involve an objective function $J(\x)$ whose minimization we seek, and where typically the excitation signals satisfy $u_1(t)=\sqrt{\omega} \;\widetilde{u}_1(\omega t)$, $u_2(t)=\sqrt{\omega} \;\widetilde{u}_2(\omega t)$ with $\widetilde{u}_1$ and $\widetilde{u}_2$ periodic (e.g., $\widetilde{u}_1(\omega t)=\cos(\omega t)$, $\widetilde{u}_2(\omega t)=\sin(\omega t)$). Then, rescaling time according to $\tau=\omega t$ and letting $\epsilon=\frac{1}{\sqrt{\omega}}$,  \eqref{eq:control_affine} becomes
\begin{equation}\label{eq:control_affine_canonical}
\frac{d\x}{d\tau}=\frac{1}{\omega}\frac{d\x}{dt}=\epsilon(\bm{b}_1(\x) \widetilde{u}_1(\tau)+\bm{b}_2(\x) \widetilde{u}_2(\tau))=\epsilon \f(\x,\tau)
\end{equation}
i.e., in averaging-canonical form \eqref{eq:average-canonical form_intro}.
As noted in \cite{durr2013lie,VectorFieldGRUSHKOVSKAYA2018}, with smooth $\bm{b}_1(\x)$, $\bm{b}_2(\x)$ and $J(\x)$ and certain control structures/laws (i.e., choices of $\bm{b}_1(\x)$ and $\bm{b}_2(\x)$), the ESC system \eqref{eq:control_affine} can be approximated using first-order Lie bracket approximations, resulting in a gradient system/flow of the following form 
$$\dot{\zz}=-\A \nabla J(\zz),$$ 
where $\A$ is positive definite. Recently, in \cite{pokhrel2023higher}, it was shown that first-order Lie bracket approximations are equivalent to second-order averaging. Hence, our aim here is to take advantage  of the results of section 3 in order to derive/analyze nonsmooth control-affine ESC systems of the form \eqref{eq:control_affine} via nonsmooth second-order averaging, which results in a generalized-gradient system/flow of the following form: 
$$\dot{\zz} = -\A \nabla_L J(\zz),$$ 
where $\nabla_L J(\zz)$ is the lexicographic gradient as mentioned in \cref{subsec:LD}.

Before establishing nonsmooth ESC results, we require the following assumptions to hold throughout this section: we suppose that (A2)-(A6) hold for the RHS function $\f$ in \eqref{eq:control_affine_canonical} (and therefore on $\bm{b}_i$ and $u_i$). Additionally, we consider the following assumptions on the excitation signals $u_i:\real\to\real$,  vector fields $\bm{b}_i:X \to \real^n$ (with open set $X$), and the behavior of the objective function $J:X\to\real$, which are common and non-restrictive  assumptions made in the ESC literature (e.g., \cite{VectorFieldGRUSHKOVSKAYA2018}):  
    \begin{enumerate}
        \item[\textbf{(B1)}] The excitation signal functions $u_1(t)$ and $u_2(t)$ are $T$-periodic in $t$ and have zero-mean, i.e., $\int_0^Tu_i(t)dt=0$ for $i=1,2$.
        \item[\textbf{(B2)}] The objective function $J(\x)$ has an isolated local minimizer\footnote{That is, $J(\x)> J^*:=J(\x^*)$ for all $\x\in D\backslash\{\x^*\}$.} $\x^*\in D$ in some compact subset $D \subset X$ and there exist $\mathcal{K}$-class functions\footnote{The function $\alpha:[0,\infty)\to[0,\infty)$ is $\mathcal{K}$-class if $\alpha(0)=0$ and $\alpha$ is strictly increasing.}  $\alpha_1$ and $\alpha_2$ such that $\alpha_1(||\x-\x^*||)\le J(\x)-J^*\le \alpha_2(||\x-\x^*||)$ for all $\x\in D$.
    \end{enumerate}

\begin{proposition}\label{prop:ESC_laws}
Suppose that assumptions (A2)-(A4) and (A6) hold. 
Let $u_1(t)=\sqrt{\omega}\cos(\omega t)$ and $u_2(t)=\sqrt{\omega}\sin(\omega t)$.
Then the following ESC law{s}
  \begin{itemize}[leftmargin=1.5cm]
     \item[(Law 1)]  $\dot{x} = J(x)u_1(t)+ u_2(t)$,
     \item[(Law 2)]  $\dot{x} = \sin(J(x))u_1(t)+\cos(J(x))u_2(t)$,
     \item[(Law 3)]  $\dot{x} = \sqrt{\frac{1-e^{-J(x)}}{1+e^{J(x)}}}(\sin(J(x)+2\ln(e^{J(x)}-1))u_1(t)+\cos(J(x)+2\ln(e^{J(x)}-1))u_2(t)$,
 \end{itemize} 
 with $J(x)\neq 0$ in Law 3, have nonsmooth second-order averaging system 
     \begin{align}\label{eq:Lyap_proof}
        \dot{\zz} = -\A \nabla_L J(\zz),
    \end{align}
    for some positive definite matrix $\A$.
 \end{proposition}

 \begin{proof} See Appendix \cref{sec:appendix}.
 \end{proof}

Next we prove that  \eqref{eq:Lyap_proof} has asymptotically stable  equilibrium $\zz^*$ (isolated local minimizer of $J$), as in the smooth gradient system case. With this result in place, since \eqref{eq:Lyap_proof} represents second-order averaging of \eqref{eq:control_affine},
 then \cref{thm:asymptotic_implies_practical} guarantees practical stability of \eqref{eq:control_affine}.  Our proof closely follows the proof outlined in \cite[Theorem 3.1]{shevitz1994lyapunov}.
 
 \begin{theorem}\label{thm:Lyap_stability}
Suppose that assumptions (A2)-(A4) and  (B2) hold.
Then  \eqref{eq:Lyap_proof} has asymptotically stable  equilibrium $\zz^*$.
 \end{theorem}
 
 \begin{proof}
     Let $V(\zz)= J(\zz)-J^*$ be a candidate Lyapanov function. 
     We see from assumption  (B2) that $V(\zz^*)=0$ and $V(\zz)>0$ for $\zz \in D \setminus \{\zz^*\}$, and that $V(\zz)$ is locally convex on a neighborhood $N_{\zz^*}$ around the  isolated minimizer $\zz^*$. It therefore follows from   \cite[Proposition 2.3.6]{clarke1990optimization} that $V$ is regular\footnote{A function is (Clarke) regular \cite{clarke1990optimization} at $\zz$ if $V'(\zz;\dd)=V^{o}(\zz;\dd)$ for all $\dd\in\real^n$, where $V'(\zz;\dd)$ is the one-sided directional derivative and $V^{o}(\zz;\dd)=\limsup_{\w\to\zz, \gamma \to 0^+} [V(\w+\gamma \dd)-V(\w)]/\gamma$ is the Clarke generalized directional derivative \cite{clarke1990optimization}.} at each $\zz\in N_{\zz^*}$. Moreover, as $V(\zz)$ is scalar-valued, it follows from \cite{nesterov2005lexicographic} that an L-gradient (corresponding to any full rank directions matrix) of $V(\zz)$ at $\zz$ is a subgradient (from convex analysis) of $V(\zz)$ at $\zz$, i.e., $(\nabla_{\rm L} V(\zz))^{\rm T} \in  \partial_{\rm L} V(\zz) \subseteq \partial_{\rm C} V(\zz)$.
      
  
     From \cite[Equation 13, Theorem 2.2]{shevitz1994lyapunov}, and given that the function $V$ is time-independent, then $\dot{V}(\zz(t))\in \dot{\Tilde{V}}(\zz(t))$ almost everywhere
     where \cite[Theorem 2.2]{shevitz1994lyapunov}:
\begin{align}\label{eq:Lyapanov_eq1}
         \dot{\Tilde{V}}(\zz(t))
       &= \bigcap_{\xi\in\partial_{\rm C} J(\zz(t))} \xi^{\rm{T}}\mathcal{F}[-\A\nabla_L J](\zz(t)),
     \end{align} 
with  
  $\mathcal{F}$ representing the Filippov set-valued map.
 As described above, since $J(\zz)$ is a scalar function $(\nabla_L J(\zz))^{\rm T} \in \partial_{\rm C} J(\zz)$. Then, based on \cite[Theorem 1, property 6]{paden1987calculus}, and since $J$ is L-smooth (and hence locally Lipschitz continuous) with $\nabla_L J(\zz)=\nabla J(\zz)$ almost everywhere, it follows that \begin{equation}\label{eq.filippovproperties}
     \mathcal{F}[-\A\nabla_L J](\zz)=-\A\mathcal{F}[\nabla_L J](\zz)=-\A\partial_{\rm C} J(\zz).
     \end{equation} 
     Thus, we have 
    \begin{align}\label{eq:Lyapanov_eq2}
         \dot{\Tilde{V}}(\zz(t)) &= \bigcap_{\xi\in\partial_{\rm C} J(\zz(t))} -\xi^{\rm{T}}\A\partial_{\rm C} J(\zz(t)).
     \end{align}
Let $U=\partial_{\rm C} J(\zz)$ (recall from \cref{subsec:LD} that $\partial_{\rm C} J(\zz)$ is convex, hence $U$ is a convex set). Define the quadratic function $g:U\to\real:\xi\mapsto\xi^{\rm{T}}\A\xi$.  We note that the convex quadratic function $g$ has a (global) minimizer on the convex domain $U$, denoted by $\xi_0$, i.e., 
 \begin{align*}
     \xi_0 = \underset{\xi\in U}{\text{arg min }}g(\xi)=\underset{\xi\in\partial_{\rm C} J(\zz(t))}{\text{arg min }}\xi^{\rm{T}}\A\xi.
 \end{align*}
 Then we see that, for any $\xi\in\partial_{\rm C} J(\zz)$, we have
\begin{align*}
       \xi^{\rm{T}}\A \xi_0  &\ge \xi_0^{\rm{T}}\A\xi_0\ge \lambda\|\xi_0\|^2,
\end{align*}
for some $\lambda>0$ (since $\A$ is positive definite), 
and so
\begin{align*}
       \xi^{\rm{T}} \A \xi  &\ge \lambda\|\xi_0\|^2.
\end{align*}
Hence, we can see from \eqref{eq:Lyapanov_eq2} that all elements in the set $\dot{\Tilde{V}}(\zz(t))$ satisfy $-\xi^{\rm{T}} \A \xi \le -\lambda\|\xi_0\|^2$, i.e., 
\begin{equation}\label{eq.tildetilde}
\dot{\Tilde{V}}(\zz(t)) \le -\lambda\|\xi_0\|^2\le 0.
\end{equation}
  Finally,  we show that $\dot{\Tilde{V}}(\zz(t)) < 0$.
Given that $\zz^*$ is an isolated local minimizer of $J$, there exists a neighborhood of $\zz^*$,  $N_{\zz^*}$, such that for any $\zz\in (N_{\zz^*}\cap D)\backslash\{\zz^*\}$, 
 $$0\notin \partial_{\rm C} V(\zz)\Rightarrow 0\notin \partial_{\rm C}J(\zz).$$ 
Then it follows from \eqref{eq:Lyapanov_eq2} that $\dot{\Tilde{V}}(\zz(t))\neq 0$ for $\zz(t) \in (N_{\zz^*}\cap D)\backslash\{\zz^*\}$. Then, from \eqref{eq.tildetilde}, it must be true that  there exists some $\rho>0$ such that $\dot{\Tilde{V}}(\zz(t)) < -\rho$  for  $\zz(t)\in (N_{\zz^*}\cap D)\backslash\{\zz^*\}$. Then the asymptotic stability of $\zz^*$ follows from \cref{thm:asymptotic_implies_practical}.
\end{proof}

\subsection{Extremum seeking control examples}\label{subsec:4.2} 

In this subsection, we provide simulation results for control-affine ESC systems of the form \eqref{eq:control_affine}, which have second-order averaged dynamics of the form \eqref{eq:Lyap_proof} -- such as Laws 1-3, in \cref{prop:ESC_laws}. We demonstrate the ability of said ESC systems to admit nonsmooth objective functions and the ability of their nonsmooth second-order averaged system to capture their behavior, including practical stability. 
We provide simulations for ESC Laws 1-3, in \cref{prop:ESC_laws}, where we use two nonsmooth objective functions, $J_1(x)$ and $J_2(x)$. These two objective functions are chosen to involve multiple compositions of smooth and nonsmooth expressions, causing the presence of multiple nonsmooth points in the domain. 
The formulas for the objective functions, as well as their L-gradients, are given in the examples below. The codes for all the simulations are available in \cite{github_averaging_ESC}.

\begin{example}\label{ex:law1_3}
    The first objective function we consider is:
\begin{equation}\label{eq:objectiveJ1}
    \begin{split}
        J_1(x)  =& |x+1| + |x-2| + |x-5|.
    \end{split}
\end{equation}
The L-gradient of $J_1(x)$ using the sum rule in \eqref{eq:sharp_rules} and the formula in \eqref{eq.LDderivative.abs} is:
\begin{equation}\label{eq:objectiveJ1_Lgrad}
    \begin{split}
        \nabla_L J_1(x)=  
        &\Bigg[\mathrm{fsign}(x+1;1)+\mathrm{fsign}(x-2;1)
        +\mathrm{fsign}(x-5;1)\Bigg].
    \end{split}
    \end{equation}
    We show in \cref{Fig:Law1_Law3_J1} simulation results of Laws 1 and 3 using the objective function $J_1(x)$ with $x_0=\Bar{y}_0=3$, $\omega=300$ and $\M=1$. Note that $J_1(x)$ is nonsmooth at both the initial point $x_0=3$ and the minimum $x^*=2$. Results show both ESC trajectories from Laws 1 and 3 converging to the minimum with practical stabilization, but Law 3 maintains vanishing oscillations. The averaged system converges asymptotically to the minimum.
\end{example}
\begin{figure}[htbp]
    \centering
    \includegraphics[width=0.7\textwidth, keepaspectratio]{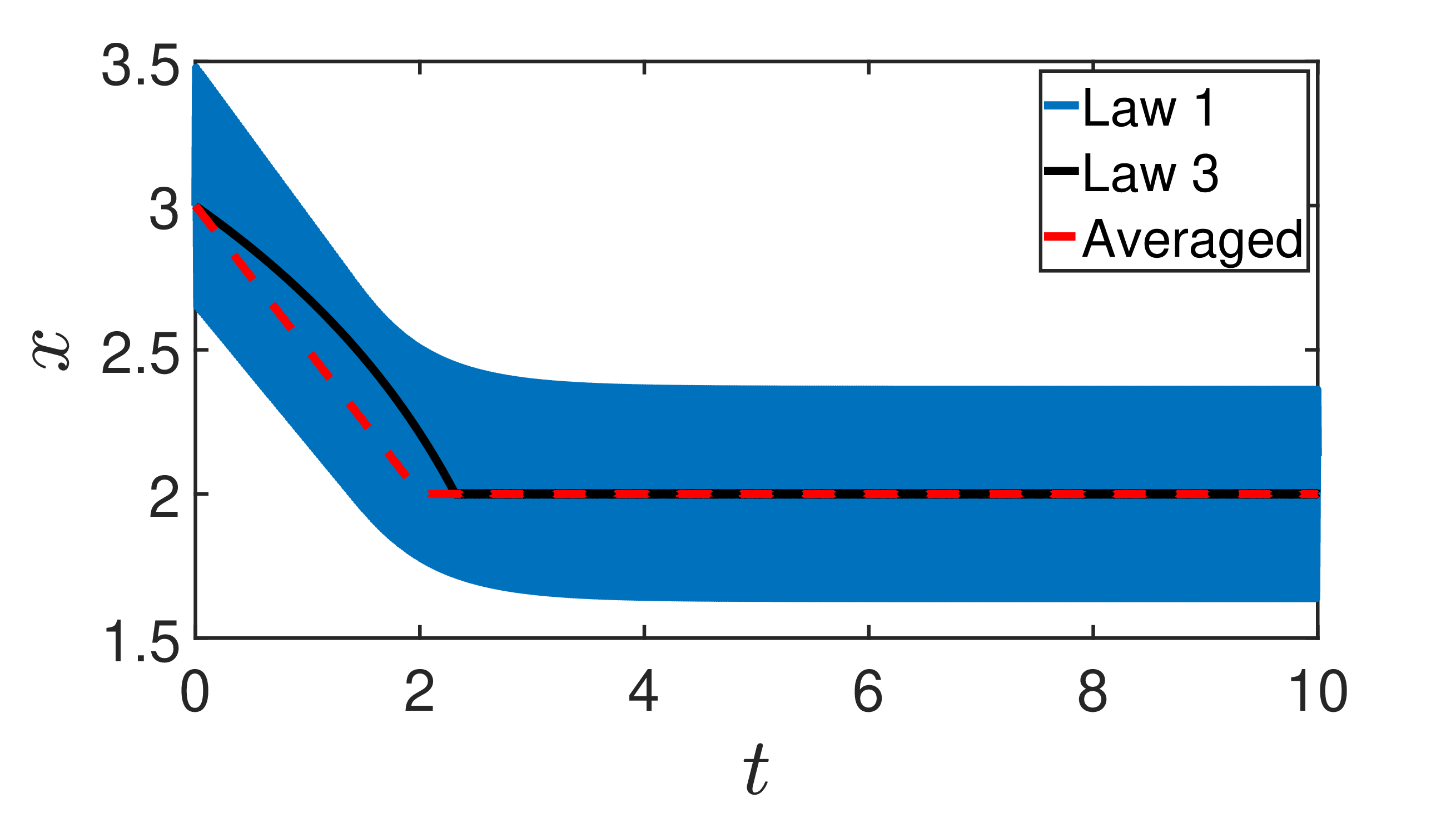}
    \caption{Simulation results for Example \cref{ex:law1_3}. ESC trajectories (Law 1 in blue and Law 3 in black) are captured by their averaged trajectory (red) while seeking minimum of $J_1(x)$.}
    \label{Fig:Law1_Law3_J1}
\end{figure}

\begin{example}\label{ex:law2ESC}
    The second objective function that we use is:
    \begin{equation}\label{eq:objectiveJ2}
    \begin{split}
        J_2(x) =& 5\min((x-2)^4,|x-2|)
        +3\min(0,x-2)+0.25x.
    \end{split}
    \end{equation}
The L-gradient of $J_2(x)$ using the chain rule in \eqref{eq:sharp_rules} and the formula in \eqref{eq:min_function_LD} is:
\small
\begin{equation*}
    \begin{split}
        &\nabla_L J_2(x)\\
        &=  
        \Bigg[0.25+5{\rm \bf slmin}\Big([(x-2)^4\quad 4(x-2)^3 ]),
        [|x-2|\quad\mathrm{fsign}(x-2,1)\Big)+3{\rm \bf slmin}\Big([0\quad 0 ]),[x-2\quad 1]\Big)\Bigg].
    \end{split}
\end{equation*}
\normalsize
We provide a simulation for Law 2 with initial condition $x_0=\Bar{y}_0=5$ using $J_2(x)$; we took $\M=1$. Hence, the ESC and the second-order averaged system will pass through two nonsmoothness and three distinct curves with three distinct derivative characteristics. We also used different $\omega$ to test the behavior. Results provided in \cref{Fig:Law2_J2} demonstrate the ability of the ESC to seek the minimum, stabilizing about it in a practical sense. Moreover, its second-order average is able to capture its behavior and converges asymptotically to the minimum. It is to be noted that with increased $\omega$, the ESC becomes better approximated by its averaged system. 
\end{example}
\begin{figure}[htbp]
    \centering
    \includegraphics[width=0.75\textwidth, keepaspectratio]{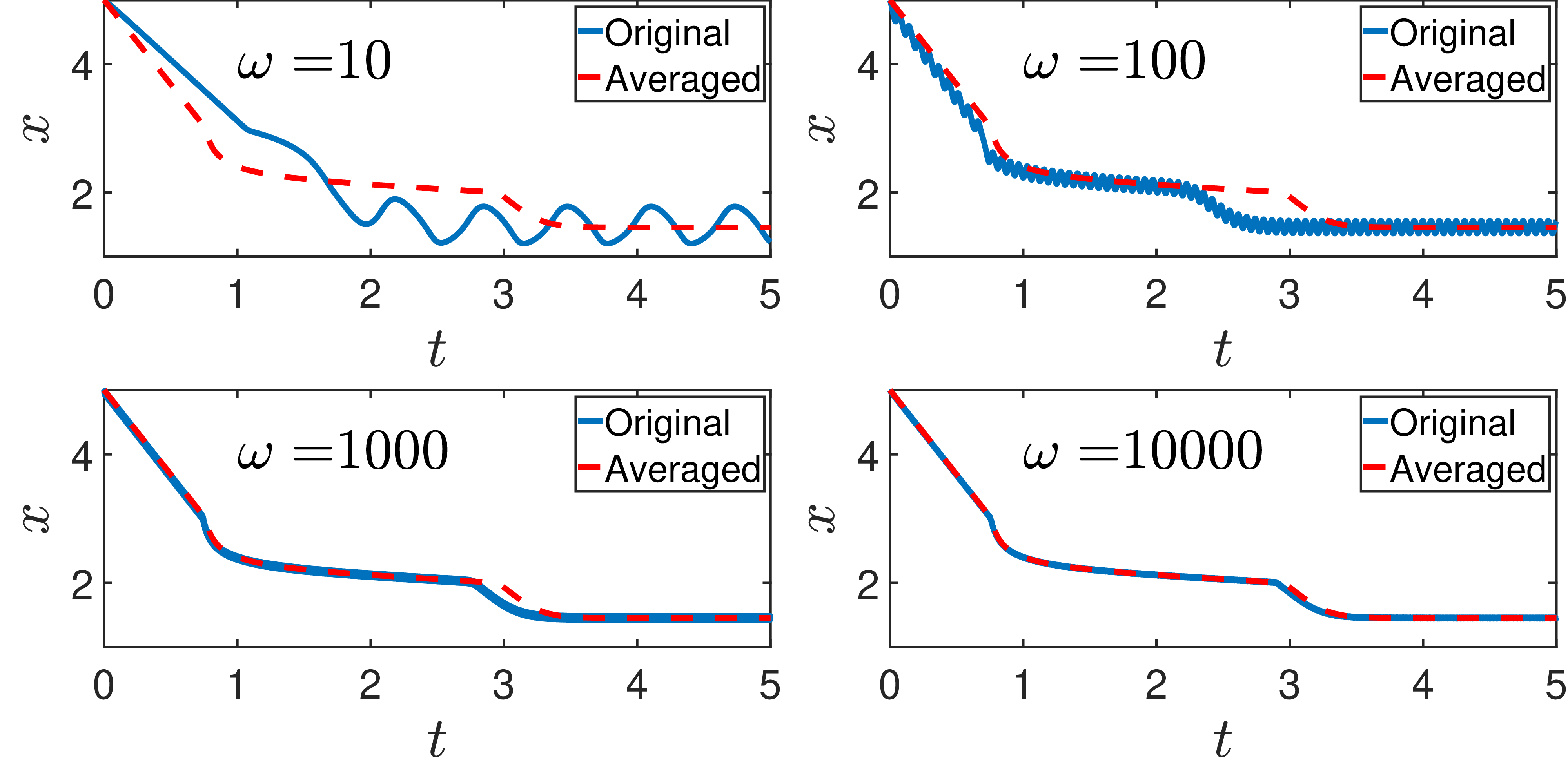}
    \caption{Simulation results for Example \cref{ex:law2ESC} for different $\omega$ values. ESC trajectories (blue) is captured by its averaged trajectories (red) while seeking minimum of  $J_2(x)$.
    }
    \label{Fig:Law2_J2}
\end{figure}
The left and right panels in \cref{Fig:objective_functions} show the objective functions \eqref{eq:objectiveJ1} and \eqref{eq:objectiveJ2} (in blue), respectively. In addition, their gradients are also shown in red. In \cref{Fig:objective_functions}, we also show the optimal points $x^*=1$ and $x^*=1.4565$, where the objective functions $J_1(x)$ and $J_2(x)$ achieve their minimum values.

\begin{figure}[ht]
    \centering
\includegraphics[width=0.7\textwidth, keepaspectratio]{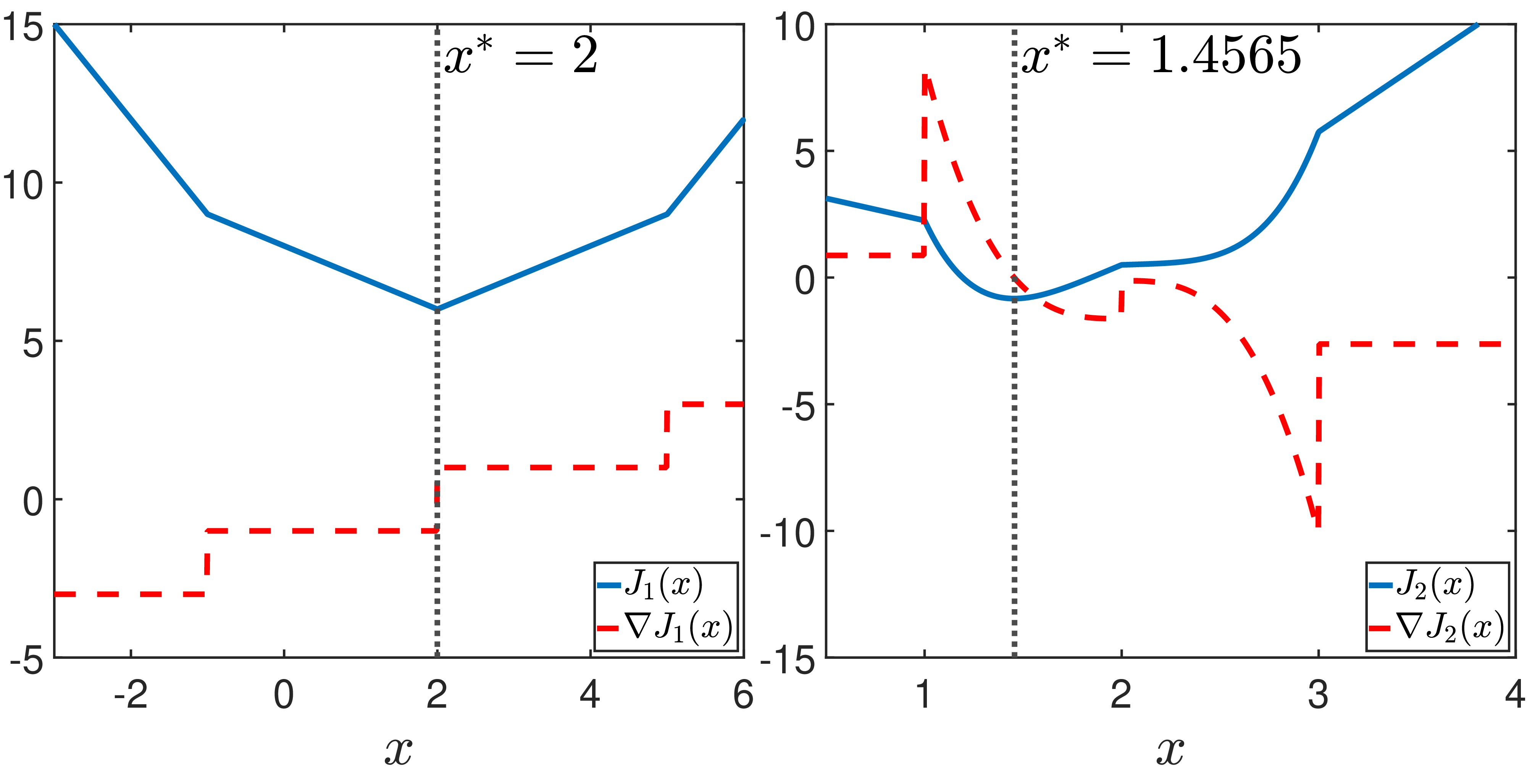}
    \caption{The left and right panels show the objective functions \eqref{eq:objectiveJ1} and \eqref{eq:objectiveJ2} (in blue), respectively. In addition, their discontinuous gradients are also shown in red.}
    \label{Fig:objective_functions}
\end{figure}

\section{Conclusion}
In this paper, we provided a higher-order nonsmooth (L-smooth) averaging theory, with formulas up to second-order averaging, that is analogous to near-identity transformation (classical) averaging. The use of lexicographic differentiation for the first time as a tool from generalized derivatives theory in averaging, was proven useful in both the theoretical analysis and the computations. These results enabled us to provide generalized-gradient-based model-free extremum seeking methods that are operable in the nonsmooth (L-smooth) setting. We demonstrated the effectiveness of many results in this paper via examples and numerical simulations.

\appendix
\section{Proof of \cref{prop:ESC_laws}}\label{sec:appendix}
\subsection{Proving Law 1}
Let $\tau=\omega t$ and $\epsilon = \frac{1}{\sqrt{\omega}}$,  then Law 1 becomes
\begin{align*}
    \frac{dx}{d\tau} = \frac{1}{w}\frac{dx}{dt}
    &= \frac{1}{\omega}(J(x)\sqrt{\omega}\cos{(\omega t)}+\sqrt{\omega}\sin{(\omega t)})  
     = \epsilon (J(x)\cos{(\tau)}+\sin{(\tau)})
    = \epsilon f(x,\tau),
\end{align*}
which is in the averaging-canonical form.
Let $T=2\pi$, and similar to the process in \cref{thm:NonsmoothTransformationValidity}, let $f_1(x,\tau) = f(x,\tau,0)$ and $f_2(x,\tau) = 2\frac{\partial f}{\partial \epsilon}\Big|_{\epsilon=0}(x,\tau,\epsilon)=0$.  Then, the first-order averaging becomes: $g_1(x) = \frac{1}{2\pi}\int_0^{2\pi} J(x)\cos{(s)}+\sin{(s)} ds = 0$.
To calculate the second-order term, $g_2(x)$, we first note that $w_1(x,\tau)=\int_{0}^{\tau} f_1(x,s) - g_1(x) ds 
    = \int_{0}^{\tau} f(x,s,0) ds$.
Using the chain rule in \eqref{eq:sharp_rules} we get
    $\nabla_L f_1(x) 
    = \cos(\tau) \nabla_L J(x)$.
Similarly, 
    $\nabla_L w_1(x) 
    =  \int_{0}^{\tau}\cos(s) \nabla_L J(x)ds$.
Therefore, we have
\begin{align*}
     \nabla_L  f_1 w_1 - \nabla_L w_1  f_1
    =& \cos(\tau) \nabla_L J(x)\int_{0}^{\tau} J(x)\cos{(s)}+\sin{(s)}ds
    -\int_{0}^{\tau} \cos(s) \nabla_L J(x)ds(J(x)\cos{(\tau)}+\sin{(\tau)})\\
    =& \cos(\tau) \nabla_L J(x)(J(x)\sin{(\tau)}-\cos{(\tau)}+1)\\
    &-\sin(\tau) \nabla_L J(x)(J(x)\cos{(\tau)}+\sin{(\tau)})
    -\sin^2(\tau) \nabla_L J(x)\\
    =&\cos(\tau) \nabla_L J(x)-\nabla_L J(x).
\end{align*}
Since $g_1(x)=0$, we see that
     $\nabla_L  g_1 w_1 - \nabla_L  w_1  g_1=0$.
 Therefore, we see that
\begin{align*}
    g_2(x)
    = &\frac{1}{T}\int_0^T  \nabla_L f_1 w_1 - \nabla_L w_1  f_1d\tau= - \nabla_L J(x).
\end{align*}
 Therefore, the second-order averaged system is 
\begin{align*}
    \frac{d\Bar{y}}{d\tau} &=\epsilon g_1 (\Bar{y})+\frac{\epsilon^2}{2!} g_2 (\Bar{y}) = -\frac{1}{2\omega} \nabla_L J(\Bar{y}).
\end{align*}
Finally, we represent the second-order averaged system in terms of $t$ as 
\begin{align}\label{eq:ex_averaged_system_Law1}
    \frac{d\Bar{y}}{dt} &= \omega \frac{d\Bar{y}}{d\tau} = -\omega \frac{1}{2\omega} \nabla_L J(\Bar{y}) = -\frac{1}{2} \nabla_L J(\Bar{y}).
\end{align}
\subsection{Proving Law 2}
We get the system in Law 2 in the averaging canonical form as before
\begin{align*}
    \frac{dx}{d\tau} 
    &= \epsilon (\sin{(J(x))}\cos{(\tau)}+\cos{(J(x))}\sin{(\tau)})=\epsilon f(x,\tau),
\end{align*}
Let $T=2\pi$, $f_1(x,\tau) = f(x,\tau,0)$ and $f_2(x,\tau) = 2\frac{\partial f}{\partial \epsilon}\Big|_{\epsilon=0}=0$.  Then, the first-order averaging becomes:
\begin{align*}
    g_1(x) &= \frac{1}{2\pi}\int_0^{2\pi} \sin{(J(x))}\cos{(s)}+\cos{(J(x))}\sin{(s)} ds = 0.
\end{align*}
Also we have 
    $w_1(x,\tau) = \int_{0}^{\tau} f(x,s,0) ds$, and using the chain rule in \eqref{eq:sharp_rules} we get
\begin{align*}
     \nabla_L f_1 w_1 - \nabla_L w_1  f_1
    =& [\nabla_L J(x)\cos{(J(x))}\cos{(\tau)}-\nabla_L J(x)\sin{(J(x))}\sin{(\tau)}]\\
    &\cdot
    \Big[\int_{0}^{\tau} \sin{(J(x))}\cos(s)+\cos{(J(x))}\sin(s) ds\Big]\\
    &-\Big[\int_{0}^{\tau} \nabla_L J(x)\cos{(J(x))}\cos{(s)}
    -\nabla_L J(x)\sin{(J(x))}\sin{(s)} ds\Big]\\
    &\qquad\cdot[\sin{(J(x))}\cos{(\tau)}+\cos{(J(x))}\sin{(\tau)}]\\
    =& \cos(\tau)\nabla_L J(x)-\nabla_L J(x).
\end{align*}
We note that $f_2(x)=0$ and 
     $\nabla_L g_1 w_1 - \nabla_L w_1 \cdot g_1=0.$ Hence
\begin{align*}
    g_2(x)
    = &\frac{1}{T}\int_0^T [\nabla_L  f_1 w_1 - \nabla_L  w_1 f_1]d\tau= - \nabla_L J(x).
\end{align*}
 Therefore, the second-order averaged system is 
\begin{align*}
    \frac{d\Bar{y}}{d\tau} &=\epsilon g_1 (\Bar{y})+\frac{\epsilon^2}{2!} g_2 (\Bar{y}) = -\frac{1}{2\omega}\nabla_L J(\Bar{y}).
\end{align*}
In terms of $t$: 
\begin{align}\label{eq:ex_averaged_system_Law2}
    \frac{d\Bar{y}}{dt} &= \omega \frac{d\Bar{y}}{d\tau}  = -\frac{1}{2}\nabla_L J(\Bar{y}).
\end{align}
\subsection{Proving Law 3}
Similar to Law 1 and Law 2, we obtain the second-order averaged system associated with Law 3. Let $\psi_1=e^{J(x)}+2\log(e^{J(x)}-1)$ and $\psi_2=\sqrt{\frac{1-e^{-J(x)}}{1+e^{J(x)}}}$, then we see (as done before) that we get the following averaging canonical form:
\begin{align*}
    \frac{dx}{d\tau} &=  \epsilon \Biggr(\psi_2\Bigr(\sin(\psi)u_1+\cos(\psi)u_2\Bigr)\Biggr)
    =\epsilon f(x,\tau).
\end{align*}
 We see that
    $g_1(x) = \frac{1}{T}\int_0^T f(x,s,0) ds=0$, and
\begin{equation}\label{eq:w1_Law3}
    \begin{split}
    w_1(x,\tau)
        &=\psi_2\,\Big(\mathrm{cos}\left(\psi_1\right)
        -\mathrm{cos}\left(t+\psi_1\right)\Big),
    \end{split}
\end{equation}
Using the chain rule in \eqref{eq:sharp_rules} we get
\begin{equation}\label{eq:JL_f_Law3}
\begin{split}
    \nabla_L  f(x,\tau)=& 
    \Bigg[\mathrm{cos}\left(\tau+\psi_1\right)\,
    \psi_2\Big({\mathrm{e}}^{J\left(x\right)} \,\nabla_L J\left(x\right)
    +\frac{2\,{\mathrm{e}}^{J\left(x\right)} \,\nabla_L J\left(x\right)}{{\mathrm{e}}^{J\left(x\right)} -1}\Big)\Bigg]\\
    &+\frac{1}{2}\Biggr[\frac{\mathrm{sin}\left(\tau+\psi_1\right)\,\frac{{\mathrm{e}}^{-J\left(x\right)} \,\nabla_L J\left(x\right)}{{\mathrm{e}}^{J\left(x\right)} +1}}{\psi_2} 
    + \frac{\mathrm{sin}\left(\tau+\psi_1\right)\,\frac{{\mathrm{e}}^{J\left(x\right)} \,{\left({\mathrm{e}}^{-J\left(x\right)} -1\right)}\,\nabla_L J\left(x\right)}{{{\left({\mathrm{e}}^{J\left(x\right)} +1\right)}}^2 }\,}{\psi_2}\Biggr]
\end{split}
\end{equation}
and
\begin{equation}\label{eq:JL_w_Law3}
\begin{split}
    &\nabla_L  w(x,\tau)\\
    &=\frac{1}{{{\left({\mathrm{e}}^{J(x)} +1\right)}}^2 \,{{\left(\frac{{\mathrm{e}}^{-J(x)} \,{\left({\mathrm{e}}^{J(x)} -1\right)}}{{\mathrm{e}}^{J(x)} +1}\right)}}^{\frac{1}{2}} }\\
    &\qquad\cdot\Biggr[\frac{1}{2}\,{\mathrm{e}}^{-J(x)} \,\Big(\mathrm{cos}\left(\psi_1\right)-\mathrm{cos}(t+\psi_1)\Big)
    \left(2\,{\mathrm{e}}^{J(x)} -{\mathrm{e}}^{2\,J(x)} +1\right)\nabla_L J(x)\Biggr]\\
    &\quad-\frac{1}{{\mathrm{e}}^{J(x)} -1}\Biggr[{\mathrm{e}}^{J(x)} \,{\left({\mathrm{e}}^{J(x)} +1\right)}\Big(\mathrm{sin}\left(\psi_1\right)-\mathrm{sin}\left(t+\psi_1\right)\Big)
    \sqrt{\frac{{\mathrm{e}}^{-J(x)} \,{\left({\mathrm{e}}^{J(x)} -1\right)}}{{\mathrm{e}}^{J(x)} +1}}\,\nabla_L J(x)\Biggr].
\end{split}
\end{equation}
 We also note that $f_2(x,\tau)=0$ and 
     $\nabla_L g_1 w_1 - \nabla_L w_1  g_1=0$, which, along with \eqref{eq:w1_Law3}, \eqref{eq:JL_f_Law3} and \eqref{eq:JL_w_Law3}, following lengthy (but systematic) calculations gives that 
\begin{align*}
    g_2(x)
    =& - \nabla_L J(x).
\end{align*}
Therefore, the second-order averaged system in terms of $t$ is:  
\begin{align}\label{eq:ex_averaged_system_Law3}
    \frac{d\Bar{y}}{dt} &= \omega \frac{d\Bar{y}}{d\tau}  = -\frac{1}{2}\nabla_L J(\Bar{y}).
\end{align}

\section*{Acknowledgments}
This material is based upon work supported by the National Science Foundation under Award No. 2318772 and 2318773.
\small
\bibliographystyle{siamplain}
\bibliography{bibliography,references}
\end{document}